\documentclass[a4paper,reqno]{amsart}
\usepackage[a4paper, total={5.35 in, 8.3 in}]{geometry}
\usepackage{amsmath,amsfonts,amssymb}

\usepackage{verbatim}
\usepackage{enumerate}
\usepackage{amsthm}
\usepackage{url}
\usepackage{multicol}
\usepackage[colorlinks]{hyperref}
\usepackage[english]{babel}
\usepackage{tikz}
\usepackage{mathtools}
\usepackage[utf8]{inputenc}
\usepackage{xcolor}

\newcommand{\Tr}{\mathrm{Tr}}

\newcommand{\F}{\mathbb{F}}

\newcommand{\fq}{\mathbb{F}_q}

\newcommand{\ordem}{\text{ord}}

\newtheorem{theorem}{Theorem}[section]

\newtheorem{proposition}[theorem]{Proposition}
\newtheorem{definition}[theorem]{Definition}
\newtheorem{problem}{Problem}
\newtheorem{lemma}[theorem]{Lemma}
\newtheorem{corollary}[theorem]{Corollary}
\newtheorem{example}[theorem]{Example}
\newtheorem{remark}[theorem]{Remark}

\author{Daniela  Oliveira and F. E. Brochero Mart\'{i}nez}

\title{%On circulant matrices and affine rational points  of  Artin-Schreier  curves \\ OR \\ 
Artin-Schreier curves given by $\fq$-linearized polynomials}

\begin{document} 
%\baselineskip=2\baselineskip

\maketitle 

\begin{abstract}

Let $\F_q$ be a finite field with $q$ elements, where $q$ is a power of an odd prime $p$.  In this paper we associate circulant matrices and quadratic forms with the  Artin-Schreier curve $y^q -y= x \cdot F(x)- \lambda,$ where $F(x)$ is a $\F_q$-linearized polynomial and $\lambda \in \F_q$. Our results provide a characterization of the number of affine rational points  of this curve in the extension $\F_{q^r}$ of $\F_q$, for $\gcd(q,r)=1$.  In  the case  $F(x) = x^{q^i}-x$ we give a complete description of the number of affine rational points  in terms of Legendre symbols and quadratic characters.
\end{abstract}

\section{Introduction}
%Let $\F_{q}$ be a finite field with $q=p^e$ elements, where $p$ is an odd prime and $e$ a positive integer. 
%Let us denote $\Tr$ the relative trace map from $\F_{q^r}$ to $\F_{q^r}.$ 
Information about the number of affine rational points  of algebraic curves over finite fields has many applications in coding theory, cryptography, communications and related areas, e.g. \cite{NiCha, TVN}. In this paper we investigate the number of affine rational points  of plane curves given by  
%It is important to know the number of affine rational points  of the curve in order to using those applications. Throughout this paper, we will work with a family of curves, that are called Artin-Schreier curves. For $g(x) \in \F_q[x]$, a plane curve with equation of the form
\begin{equation}\label{artin}
\mathcal C_{g} : y^q - y =g(x)
\end{equation}
in extensions $\F_{q^r} $ of $\F_q$ where $q=p^e$, $p$ is an odd prime, $r \in \mathbb N$ and $g(x) \in \F_q[x]$. These curves are called  Artin-Schreier curves and have been extensively studied  in several contexts, e.g.  \cite{Coulter, HefezKakuta,OzbudakSaygi,G.vanGeer1}. 

%Artin-Schreier curves are closely related , as follows. In general, 
For  a polynomial $g(x) \in \F_q[x]$ and $\F_{q^r}$ a finite extension of $\F_q$  we can associate the map
\begin{equation}\label{quadratic}
\begin{matrix}
Q_g:&
\F_{q^r} &\to &\F_q\\
&\alpha&\mapsto &\Tr(g(\alpha)),
\end{matrix}
\end{equation}
where $\Tr:\F_{q^r}\to\F_q$ denotes the trace function. %, i.e., $\Tr(\alpha) =\displaystyle \sum_{i=0}^{r-1}\alpha^{q^i}$ for $\alpha \in \F_{q^r}$.
Let $N_r(Q_g)$ denote the number of zeroes of $Q_g$  in $\F_{q^r}$ and $N_r(\mathcal C_{g})$ the number of affine rational points  of $\mathcal C_{g}$ over $\F_{q^r}^2$. From Hilbert's Theorem $90$ we have 
%\[N(Q) = |\{x \in \F_{q^r} \mid Q(x) = 0 \}|\]
%and $N(\mathcal C)$ the number of affine affine rational points  of the curve $\mathcal C$ in $\F_{q^r}$, using the 
\[N_r(\mathcal C_{g}) = qN_r(Q_g).\]
It follows that the determination of $N_r(\mathcal C_{g})$ is equivalent to the determination of $N_r(Q_g)$.
%Therefore determining $N(Q)$ it is equivalent to determining $N(\mathcal C)$.
Details can be found in 
%Some characterizations and classifications  about these relations can be found  in 
\cite{AnbarMeidl,AnbarMeidl2,OzbudakSaygi2}.

In \cite{Wolfmann}, Wolfmann  determined the number of rational points of the algebraic plane curve defined over $\F_{q^k}$ by the equation $y^q - y = ax^s + b$ where $a\in \F_{q^k}^*$ and $b\in \F_{q^k}$, for even $k$ and special integers $s$.  
In \cite{Coulter}, Coulter determined the number of $\F_q$-rational points of the curve $y^{p^n}-y=ax^{p^{\alpha}+1} + L(x), $  where $a \in \F_q^*$, $t=\gcd(n,e)$ divides $d=(\alpha,e)$ and $L(x)\in \F_q[x]$ is a $\F_{p^t}$-linearized polynomial. In this paper, we determine  $N_r(\mathcal C_{g})$ for  some families of  Artin-Schreier  curves given by specific  polynomials $g(x) \in \F_q[x].$ 
%In our case, we use the relations between the curve $\mathcal C_{g}$ and the map $Q_g$ to given an explicit expression for the number of affine rational points  of the curve $\mathcal C_{g}$, when $g(x)$ satisfies suitable conditions. 

The first aim of this paper is to find $N_r(\mathcal C_{g})$ when  $g(x)=x  F(x)-\lambda$, where $F(x)$ is a $\F_q$-linearized polynomial, $\lambda \in \F_q$ and we assume $\gcd(r,p)=1$. In this cases, we denote $\mathcal C_g$ by $\mathcal C_{F,\lambda}$ for $F(x)$ and $\lambda$ fixed. 
 In order to do so, we prove that $Q_{g+\lambda}$ defines a quadratic form and use this form  to find a connection between the number of affine rational points with the rank of an appropriate circulant  matrix. Theorem~\ref{thF(x)}  provides an explicit formula for $N_r(\mathcal C_{F,\lambda})$ in this case.

A second point of this paper, assuming the hypothesis of Theorem \ref{thF(x)}, is to study the case  $F(x) = x^{q^i}-x$ when $i$ is a positive integer, i.e.,   $\mathcal C_{F,\lambda}$ defined by 
$$y^q-y = x^{q^i+1} - x^2 - \lambda. $$
%We assume that $\gcd(r,p) = 1$ and 
In Theorem \ref{th2}  we find an expression of $N_r(\mathcal C_{x^{q^i}-x,\lambda})$  in terms of Legendre symbols and $p$-adic valuations.

When $i=1$ and $\lambda =0 $  we obtain the curve  $y^q-y = x^{q+1} - x^2 $, which is associated to the number of monic irreducible polynomials in $\F_q[x]$ of degree $r$ with prescribed first two coefficients, i.e., when  we fix the coefficients of $x^{r-1}$ and $x^{r-2}$. To see this,  consider
the map
\begin{equation}
\begin{matrix}
T_2:&
\F_{q^r} &\to &\F_q\\
&\alpha&\mapsto &\sum\limits_{0 \le i < j \le r-1} \alpha^{q^i+q^j}.
\end{matrix}
\end{equation}
 $\Tr(\alpha)$ and $T_2(\alpha)$  determine, respectively, the coefficients of $x^{r-1}$ and $x^{r-2}$ of the characteristic polynomial of $\alpha\in\F_{q^r}$ over $\F_q$.  A straightforward calculation shows that  $T_2(\beta^q-\beta) = \Tr(\beta^{q+1}-\beta^2)$ for all $\beta \in \F_{q^r}$. % For $\beta \in \F_{q^r}$,  
By Hilbert's  Theorem $90$ we have that  $\Tr(\beta)=0$ if and only if there exists $\alpha \in \F_{q^r}$ such that $\beta=\alpha^q-\alpha$ and therefore \(\Tr(\beta) =0 \text{ and } T_2(\beta)=0\text{ if and only if } 0 =T_2(\beta) = T_2(\alpha^q-\alpha) = \Tr(\alpha^{q+1}-\alpha^2) .\)
Consequently, the number of irreducible polynomials of degree $r$ with first two coefficients being zero can be related to the number  of affine rational points  of  the curve $y^q-y= x^{q+1}-x^2$. For more details see \cite{coefprescribed2,coefprescribed1}.

This paper is organized as follows. Section $2$ provides  background material and preliminary results.  In Section $3$ we discuss the case  $g(x) = x F(x)- \lambda$ where $F(x)$ is a $\F_q$-linearized polynomial and $\lambda \in \F_q$. In Section $4$ we give an explicit formula for $N_r(C_{x^{q^i}-x,\lambda})$ when $\gcd(r,p)=1$. 
%We finish in Section 5 by giving an alternative expression  for $N_r(\mathcal C_{x^{q^i}-x,\lambda})$  when $i=1$, which  includes the case where  $\gcd(r,p) =p$.

\section{Preliminary results}

Throughout this paper, $\F_{q}$ denotes a finite field with $q$ elements, where $q$ is a power of an odd prime $p$ and $r$ is a fixed positive integer.  % For $r$ a positive integer 
We % and the field $\F_{q^r}$ we 
define the trace function 
\begin{align*}
\Tr_{\F_{q^r}/\F_q}:\F_{q^r} &\to \F_q\\
\alpha&\mapsto \alpha+\alpha^q+\cdots + \alpha^{q^{r-1}}
%\[\Tr_{\F_{q^r}/\F_q}(\alpha) = \alpha+\alpha^q+\cdots + \alpha^{q^{r-1}}.\]
\end{align*}
and, for simplicity,  we denote  the trace function $\Tr_{\F_{q^r}/\F_q}$  by $\Tr$.
A polynomial $F(x)\in \F_q[x]$ is called \emph{$\F_q$-linearized} if  is of the form $a_0x+a_1x^q+a_2x^{q^2}+\dots +a_l x^{q^l}$, where $a_j\in \F_q$ for all $0 \le j\le l$.   The polynomial $f(x)= a_0+a_1x+a_2x^2+\dots +a_l x^{l}$ is called \emph{associated polynomial of $F(x)$.}

In what follows, %$g(x)\in \F_q[x]$ denotes a polynomial of the form $xF(x)-\lambda$, 
for $F(x)$  a  $\F_q$-linearized and $\lambda\in \F_q$, $\mathcal C_{F,\lambda}$  denotes the curve determined by the equation 
\begin{equation} \label{equacao1}
y^q-y=xF(x)-\lambda
\end{equation}
and $Q_{xF(x)}: \F_{q^r} \to \F_{q^r}$  is the quadratic form given by  $Q_{xF(x)}(\alpha) = \Tr(\alpha F(\alpha))$.
We observe that if $(\alpha,\beta)\in \F_{q^r}^2$ is  a point of  $\mathcal C_{F,\lambda}$, i.e., $\beta^q-\beta=\alpha F(\alpha)-\lambda$ then 
$$0=\Tr(\beta^q-\beta)=\Tr(\alpha F(\alpha)-\lambda)=\Tr(\alpha F(\alpha)) -r\lambda.$$
Reciprocally, if $\alpha\in \F_{q^r}$ satisfies the equation $\Tr(\alpha F(\alpha)) =r\lambda$, then $\Tr(\alpha F(\alpha)-\lambda)=0$ and by Hilbert's Theorem $90$  there exists $\beta\in \F_{q^r}$ such that  $\beta^q-\beta=\alpha F(\alpha)-\lambda$. In addition, any other solution is of the form $\beta+c$  for  $c\in \F_q$ and  it follows that 
\begin{equation}\label{Npontos}
N_r(\mathcal C_{F,\lambda}) = qN_r(Q_{xF(x)-\lambda}).
\end{equation}
We now recall the following standard definitions.
% \vm{Acho que fica  melhor Se $V$ \'e um espa\c co vetorial sobre  $K$ e $\Phi:V\rightarrow K$ \'e uma forma quadr\'atica, a defini\c c\~ao habitual  \'e $\operatorname{dim}\Phi=\operatorname{dim}_K V$. Se $S$ \'e um subespa\c co de $V$, o \emph{ortogonal de} $S$ \'e o subespa\c co $S^{\perp}:=\{x\in V: B(x,S)=0\}$, onde $B$ \'e a forma bilinear associada a $\Phi$; em particular, $V^{\perp}$ \'e dito o \emph{radical} de $(V,B)$. O teorema geral \'e que a matriz associada a $B$ \'e  invers\'\i vel se e somente se $V^{\perp}=\{0\}$ e, nesse caso, $\operatorname{dim}S + \operatorname{dim} S^{\perp}= \operatorname{dim} V$. Tudo isso pode ser achada em T. Y. Lam (meu orientador) \emph{The algebraic theory of quadratic forms}, proposi\c c\~oes 1.2 e 1.3.
%No caso, est\'a-se definindo $\mathcal L=\F^{\perp}$; dessa maneira, definir $\operatorname{dim}\Phi=\operatorname{dim}\mathcal L$ n\~ao est\'a de acordo com a literatura habitual e pode confundir aqueles que conhecem o b\'asico da teoria de formas quadr\'aticas. Sugiro achar um novo nome para essa grandeza; n\~ao encontrei nada na literatura.}

\begin{definition}
Let $\mathbb L$  be a field and $\F$ a finite extension of $\mathbb L$, where char$(\mathbb L) \neq 2$. For a quadratic form  $\Phi:\F \to \mathbb L$ we define the symmetric bilinear form $\varphi:\F\times \F \to \mathbb L$ associated to  $\Phi$ by $\varphi(\alpha,\beta) = \frac{1}2\left(\Phi(\alpha+\beta) -\Phi(\alpha)-\Phi(\beta)\right)$. 
The radical of symmetric bilinear form $\Phi: \mathbb F \to \mathbb L$  is:
\[\text{rad}(\Phi) = \{\alpha \in \mathbb F : \varphi(\alpha, \beta) = 0 \text{ for all } \beta\in \F \}.\]
If rad($\Phi) = \{0\}$, $\Phi $ is non-degenerate.

Let $\mathcal B=\{v_1,\dots, v_r\}$ be a basis of $\F$ over $\mathbb L$. The $r \times r$ matrix $A=(a_{ij})$ defined by 
$$a_{ij}= \begin{cases} \Phi(v_i),& \text{if $i=j$}\\
\frac 12(\Phi(v_i+v_j)-\Phi(v_i)-\Phi(v_j)),& \text{if $i\ne j$}.
\end{cases}
$$
is the associated matrix of the quadratic form $\Phi$ in the basis $\mathcal B$. In particular, the dimension of  $\text{rad}(\Phi)$  is equal to $r- rank (A).$ 

Let $\Phi: \fq^m \to \fq$ and $\Psi : \fq^n\to \fq$ be quadratic forms where $m \ge n$. Let $A$ and $B$ be associated matrix of $\Phi$ and $\Psi$, respectively. We say that $\Phi$ is equivalent to $\Psi$ if there exists $M \in GL_m(\fq)$ such that 
\[M^T AM =\left(
\begin{array}{c|c}
B & 0 \\ \hline
 0 & 0
\end{array}\right)\in M_m(\fq),\]
where $ GL_m(\fq)$ denotes the $m\times m$ invertible matrices over $\fq$ and $M_m(\fq)$ denotes the $m\times m $ matrices over $\fq$. 
Furthermore, $\Psi $ is called a reduced form of $\Phi$ if $\text{rad}(\Psi)= \{0\}. $
\end{definition}

\begin{remark}\label{mm}
Let $F(x)$ be $\F_q$-linearized and $r$ a positive integer. It can be easily seen that the map  $\tilde \Phi: \F_{q^r} \to \F_{q^r}$ given by $\alpha\mapsto \alpha F(\alpha)$ is a quadratic form. Furthermore, since $\Tr$ is a $\F_q$-linear form,  we have that  the map  $ \Phi: \F_{q^r} \to \F_q$  given by $\Phi(\alpha)=\Tr(\tilde \Phi(\alpha))$ is  also a quadratic form.
In fact,  for all $c \in \F_q$ and $\beta \in \F_{q^r}$ we have
\begin{align*}
\Tr(c \beta F(c \beta)) = \Tr(c^2  \beta F( \beta)) = c^2 \Tr(\beta F( \beta)),
\end{align*}  hence  $\Tr((x+y)F(x+y))- \Tr(xF(x))-\Tr(yF(y))$ defines a symmetric billinear form. 

\end{remark}

The following theorem is a well known result about the number of the  solutions of  a special equation over finite fields.
\begin{theorem}[\cite{LiNi}, Theorems $6.26$ and $6.27$] \label{sol}
Let $\Phi$ be a quadratic form over $\F_{q^r}$. Let $\varphi$ be the bilinear symmetric form associated to  $\Phi$,  $v = \dim( \text{rad}(\Phi))$  and $\Psi$ a reduced nondegenerate quadratic form equivalent to $\Phi$.  We define  $S_{\lambda}  = |\{ x \in \F_{q^r}| \Phi(x) = \lambda\}|$ and denote by  $\Delta$ the determinant of the matrix associated to $\Psi$ and $\chi$   the canonical quadratic character of $\F_{q}$. 
\begin{enumerate}[(i)]
\item If  $r+v$ is even, then 
\begin{equation} \label{casopar}
S_{\lambda} =
 \left\{
\begin{array}{ll}  
q^{r-1} +Dq^{(r+v-2)/2}(q-1) \quad &\text{ if } \lambda = 0;\\
 q^{r-1} - Dq^{(r+v-2)/2} \quad &\text{ if } \lambda \neq 0,
\end{array}\right.
 %q^{r-1} +\epsilon_{\lambda} Dq^{(r+v-2)/2}
\end{equation}
where $D = \chi((-1)^{(r-v)/2}\Delta)$. 

\item If $r+v$ is odd, then
\begin{align} \label{casoimpar}
S_{\lambda} = \left\{
\begin{array}{ll}  
q^{r-1}  \quad &\text{ if } \lambda = 0;\\
q^{r-1} + Dq^{(r+v-1)/2} \quad &\text{ if } \lambda \neq 0.
\end{array}\right.
\end{align}
where $D = \chi((-1)^{(r-v-1)/2}\lambda\Delta)$.
\end{enumerate}
\end{theorem}
Clearly, in order to determine $S_{\lambda}$ we need to calculate the dimension of the radical of $\Phi$ and the determinant of a reduced matrix associate to $\Phi$. For this, important tools for our work are complete homogeneous symmetric
polynomials and circulant matrices, which we define below.

\begin{definition}\begin{enumerate}[a)]

\item The \emph{complete homogeneous symmetric polynomials} of degree $k$ is defined  by 
$$h_k(x_1, \dots, x_r)= \displaystyle \sum_{1 \le i_1 \le  \cdots \le i_k \le r} x_{i_1} \cdots x_{i_k}.$$
We denote this polynomial by $h_k(r)$. 

\item  Let  $a_0, a_1, \dots, a_{r-1}$ be  elements of a field $\mathbb L$. The $r \times r $  { \em circulant matrix} $C= (c_{ij})$ associated to the $r$-tuple $(a_0, a_1, \dots, a_{r-1})$ is given by    $c_{ij} =a_k $ for each pair $(i,j)$ such that  $j-i \equiv k \pmod r $. We denote by $C(a_0, a_1, \dots, a_{r-1})$ this circulant matrix, and the vector $(a_0,a_1,\dots, a_{r-1})$ is called  the \emph{generator vector} of the matrix $C$.
\item The \emph{associated polynomial} of the circulant matrix  $C(a_0,a_1, \dots, a_{r-1})$ is  $f(x) =\displaystyle \sum_{i=0}^{r-1} a_ix^i.$
\end{enumerate}
\end{definition}

We will show that, with some additional hypotheses (see Proposition \ref{quadform}),  there exists a base of $\F_{q^r}$ over $\F_q$  such that the associated matrix of the quadratic form $\Tr(Q_{g+\lambda})$ in this base is circulant.

The following theorem, that can be found  as an exercise  in  \cite{ComHomSymPol}, describes  another representation of the polynomial $h_k$.

\begin{theorem}[\cite{ComHomSymPol}, Ex. $7.4$]  \label{chsp}
The polynomial $h_k(r)$  can be expressed as 
  $$h_k(r) = \sum_{l=1}^r \frac{x_l^{r+k-1}}{\prod_{m=1 \atop m \neq l}^r (x_l-x_m)}.$$
\end{theorem}

These polynomials will be useful  to determine the rank of some circulant matrices. In the case when $r$ is relatively prime with the characteristic of the field $\F_q$,  it is well known (e.g. \cite{RankCirculantMatrices}) that  the determinant of  any circulant matrix $C = C(a_0,a_1, \dots, a_{r-1})$  satisfies the relation 
\[\text{det } C= \prod_{i=1}^{r} (a_0 + a_1 \omega_i + \cdots +a_{r-1}\omega_i^{r-1}),\]
where $\omega_1, \dots, \omega_r$ are the $r$-th roots of unity in some extension of $\F_q$.  We will use this fact in order to determine the rank of $C$. More precisely, the rank of $C$ is  equal to the number of  common roots of $f(x)$ and $x^r-1$, as we prove in Theorem \ref{rank}. Before we prove this we need the following definition.%calculate the rank of $C$, we given the following definitions.
\begin{definition}
For each $0 < j \le k$ integers, $A_k$ and $A_{k,j}$ denote the polynomials 
\begin{enumerate} 
\item $A_k(x_1, \dots , x_k) = \displaystyle\prod_{1 \le t < s \le k } (x_s -x_t)$, for all $ k \ge 2.$
\item $A_{k,j}(x_1, \dots , x_k)= (-1)^{j+1}\displaystyle\prod_{1 \le t < s \le k  \atop t,s \neq j } (x_s -x_t),$ for all $k\ge3$.
\end{enumerate}
\end{definition}
The following lemmas show some  relations between the complete homogeneous symmetric polynomials and  the polynomials  $A_k$ and $A_{k,j}$.%, that will be useful to determine the rank of the matrix $C$.

\begin{lemma}\label{relation}
Let $k$ be a positive integer and  for each $1\le j \le k$, $h_{r,j}(k) $ denotes the polynomial $h_r(x_1, \dots , \hat{x}_j, \dots, x_k), $ where $\hat{x}_j$ means that the variable $x_j$ is omitted. Then $$\displaystyle \sum_{j=1}^{k} x_j^{r-1}A_{k,j}h_{r-k+1, j}(k)=0, \text{ for all } k \ge 3.$$
\end{lemma}

\begin{proof}
 Let us denote $\epsilon_{l,j} = \begin{cases}
1 & \text{ if } l>j\\
-1 & \text{ if } l<j\\
0 & \text{ if } l=j
\end{cases}$. By Theorem \ref{chsp} it follows that 

\begin{align*}
\displaystyle \sum_{i=1}^{k} x_j^{r-1} A_{k,j} h_{r-k+1,j}(k)& = \sum_{j=1}^{k} x_j^{r-1} (-1)^{j+1} \prod_{1 \le t<s \le k \atop t,s \neq j}(x_s-x_t) \sum_{l=1 \atop l \neq j}^{k} \left( \frac{x_l^{r-1}}{\prod_{m=1 \atop m \neq j,l }^{k} (x_l-x_m)}\right)\\
%& = \displaystyle \sum_{j=1}^{k} x_j^{n-1} (-1)^{j+1}  \sum_{l=1 \atop l \neq j}^{k} \left( \frac{x_l^{n-1}\displaystyle\prod_{1 \le t<s \le k \atop t,s \neq j}(x_s-x_t)}{\displaystyle\prod_{m=1 \atop m \neq j,l }^{k} (x_l-x_m)}\right)\\
& = \displaystyle \sum_{j=1}^{k} x_j^{r-1}  (-1)^{j+1}  \sum_{l=1 \atop l \neq j}^{k} x_l^{r-1}\displaystyle\prod_{1 \le t<s \le k \atop t,s \neq j,l }(x_s-x_t)(-1)^{k-l}\epsilon_{l,j} \\
& = \displaystyle \sum_{j=1}^{k}  \sum_{l=1 \atop l \neq j}^{k}   x_j^{r-1}x_l^{r-1}\displaystyle\prod_{1 \le t<s \le k \atop t,s \neq j,l }(x_s-x_t)(-1)^{k+j-l+1}\epsilon_{l,j} .\\
\end{align*}
For each $l$ and $j$ fixed, the sum runs over the term $(x_lx_j)^{r-1} = (x_jx_l)^{r-1}$ twice and then
$$ x_j^{r-1}x_l^{r-1}\left(\displaystyle\prod_{1 \le t<s \le k \atop t,s \neq j,l }(x_s-x_t)(-1)^{k+j-l+1}(\epsilon_{l,j} + \epsilon_{j,l} ) \right) = 0.$$ 
\end{proof}

\begin{lemma} \label{relationAk}
Let $k \ge 2$ be an integer. Then 
$$A_{k+1} = \sum_{j=1}^{k+1} \frac{x_1 \dots x_{k+1}}{x_j} A_{k+1,j}.$$ 
In addition,
$$\sum_{j=1}^{k} \frac{x_1 \cdots x_{k}}{x_j} \dfrac{1}{\prod_{r =1 \atop r \neq j}^{k} (x_r-x_j)}=\frac {F(x_1,\dots , x_{k+1} )}{A_{k+1}}= 1. $$ 

\end{lemma}

\begin{proof}
We now let
 $$F_{k+1} = \sum_{j=1}^{k+1} \frac{x_1 \dots x_{k+1}}{x_j} A_{k+1,j}.$$
We will prove that $A_{k+1}=F_{k+1}$   by induction on the number of variables. For $k=2$ we have
\begin{align*}
F_3 & = \sum_{j=1}^{3} \frac{x_1x_2x_3}{x_j}(-1)^{j+1} \prod_{1 \le s < t \le 3 \atop s,t \neq j}(x_t-x_s) \\
& = x_2x_3(x_3-x_2) - x_1x_3(x_3-x_1) + x_1x_2(x_2-x_1) \\
& = (x_2-x_1)(x_3-x_1)(x_3-x_2) = A_3.
\end{align*}
Now suppose that the result is true for  some $k\ge 2$. The polynomial   $F_{k+1}$  has degree ${k+1 \choose 2}$ and if $x_i = x_j$ for any $1 \le i < j \le k+1$, it follows that $F_{k+1}= 0$, which implies that $A_{k+1}$ divides  $F_{k+1}$. Since $A_{k+1}$ has the same degree of $F_{k+1}$ we obtain that $F_{k+1} = c A_{k+1}$ for some $ c \in \F_q$. From the fact that $A_{k+1}$ is a monic polynomial in respect to the variable $x_{k+1},$  it remains to show that $c=1$. In order to prove this,  we equate the coefficients of  $x_{k+1}^k$ on both sides of $F_{k+1} = c A_{k+1}$ to get 
\[\sum_{j=1}^{k} \frac{x_1 \cdots x_k}{x_j} (-1)^{j+1}\prod_{1 \le s < t \le k \atop s,t \neq j}(x_t-x_s)= c \cdot \prod_{1 \le s < t \le k}(x_t-x_s) \]
and this means that $F_k = cA_k$. 
By the induction hypothesis, it follows that $c=1$.
\end{proof}

%\begin{remark}  \label{rem1}
%From Lemma \ref{relationAk}, it follows easily that
%$$\sum_{j=1}^{k} \frac{x_1 \cdots x_{k}}{x_j} \dfrac{1}{\prod_{r =1 \atop r \neq j}^{k} (x_r-x_j)}=\frac {F(x_1,\dots , x_{k+1} )}{A_{k+1}}= 1. $$ \vm{Essa observa\c c\~ao ser\'a importante mais tarde? Se sim, n\~ao deveria ser um 'remark'.}
%\end{remark}

%{\color{purple} Posso deixar comentado?
%\begin{align*}
%1 =\dfrac{\sum_{j=1}^{k}\frac{\alpha_1\cdots\alpha_k}{\alpha_j}(-1)^{j+1} \prod_{1 \le s < t \le k \atop t,s \neq j}(\alpha-t-\alpha_s)}{\prod_{1 \le s < t \le k }(\alpha_t - \alpha_s)} = \sum_{j=1}^{k} \frac{\alpha_1 \cdots \alpha_k}{\alpha_j} \frac{1}{\prod_{r=1 \atop r \neq j}^k (\alpha_r-\alpha_j) }
%\end{align*}
%}

\begin{lemma} \label{innerproduct}
Let $C$ be a  circulant matrix  over $\F_q$ with generator vector $(a_0,a_1,\dots, a_{r-1})$ and  $f(x)$ be the  associated polynomial to the matrix $C$. %, where $\gcd(r,p)= 1$. 
Let $g(x) = \gcd(f(x), x^r-1) $ and  $\alpha_1, \alpha_2, \dots, \alpha_m$  the roots of $g(x)$. 
If $g(x)$ has only simple roots, then for each positive integer $j \le m$ the  relation  
\begin{align} \label{eq1}
(a_0,a_1,\dots ,a_{r-j}) \boldsymbol{\cdot} \left(1, h_{1}(\alpha_1, \dots, \alpha_j), \dots , h_{r-j}(\alpha_1, \dots, \alpha_j)\right)= 0
\end{align}
 is satisfied, where $\boldsymbol{\cdot}$ denotes the inner product.

\end{lemma}
\
\begin{proof} We set $\vec{\alpha\!\!\!\!\alpha}_k =(\alpha_1, \dots, \alpha_{k+1})$.
%Since $\gcd(r,q)=1$ it follows that $\alpha_i \neq \alpha_j$ for all $ i\neq  j.$  
We proceed  by induction on the number of roots of $g$. For any $\alpha$  root of $g$ we have
\[a_0+a_1 \alpha + \cdots + a_{r-1}\alpha^{r-1} = 0 .\]
Then $(a_0, a_1, \dots, a_{r-1}) \boldsymbol{\cdot}(1, \alpha, \dots, \alpha^{r-1}) = 0$  and this relation is equivalent to  the first case of the induction. 
If $j=2$, for each pair of roots $\alpha_1$ and $\alpha_2$,  we have the  relations 
\[ \begin{cases}
a_0\alpha_1+a_1 \alpha_1^2 + \cdots + a_{r-1}\alpha_1^{r} = 0 \\
a_0\alpha_2+a_1 \alpha_2^2 + \cdots + a_{r-1}\alpha_2^{r} = 0.
\end{cases}\]
Subtracting, we get
$$a_0(\alpha_2-\alpha_1)+a_1 (\alpha_2^2 -\alpha_1^2)+ \cdots + a_{r-2}(\alpha_2^{r-1}-\alpha_1^{r-1} )= 0.$$  Since $A_2 = \alpha_2 - \alpha_1 \neq 0$ it follows that
\begin{align*} 0 & =(a_0,a_1,\dots, a_{r-2}) \boldsymbol{\cdot} \left(\alpha_2-\alpha_1, \alpha_2^2-\alpha_1^2, \dots, \alpha_2^{r-1}-\alpha_1^{r-1}\right) \\
& = (a_0,a_1,\dots, a_{r-2}) \boldsymbol{\cdot} A_2\left(1,  h_{1}(\alpha_1,\alpha_2),\dots, h_{r-2}(\alpha_1,\alpha_2)\right)\end{align*} 
and this relation proves  the case and we have the case $j=2$.
Let us suppose now that  \eqref{eq1} is true for any choice of $k$ different roots of $g$ and let  $\alpha_1, \dots, \alpha_{k+1}$ be $k+1$  roots of $g$. By the induction hypothesis, we have $k+1$ equations of the form \eqref{eq1}, where for each one we do not consider one of the roots, i.e., the $j$-th equation is given by
\begin{align} \label{eq2}
(a_0, \dots, a_{r-k}) \boldsymbol{\cdot} \left(1, h_{1,j}(\vec{\alpha\!\!\!\!\alpha}_{k+1}),\dots, h_{r-k,j}(\vec{\alpha\!\!\!\!\alpha}_{k+1}) \right)=0.
\end{align}

Multiplying  the vector $ \left(1, h_{1,j}(\vec{\alpha\!\!\!\!\alpha}_{k+1}),\dots, h_{r-k,j}(\vec{\alpha\!\!\!\!\alpha}_{k+1}) \right)$ for $\alpha_j^{r-1}A_{k+1,j}$ and adding these vectors we obtain the  vector
$$ \vec{u} = \sum_{j=1}^{k+1} \alpha_j^{r-1}  \left( A_{k+1,j}, A_{k+1,j}h_{1,j}(\vec{\alpha\!\!\!\!\alpha}_{k+1}),\dots,  A_{k+1,j} h_{r-k, j}(\vec{\alpha\!\!\!\!\alpha}_{k+1}) \right).$$
%And this implies 
%$$(a_0, \dots, a_{n-k}) \cdot \left(A_{k+1,j}, A_{k+1,j}B_{2,j},\dots, A_{k+1,j}B_{n-k+1,j} \right)=0.$$
By Lemma \ref{relation}, the last coordinate of $\vec u$ is
\begin{align}\label{k+1}
 \displaystyle \sum_{j=1}^{k+1} \alpha_j^{r-1} A_{k+1,j}h_{r-k,j}(\vec{\alpha\!\!\!\!\alpha}_{k+1}) =0.
\end{align}
Let us put  $\alpha = \alpha_1 \cdots \alpha_{k+1}$. The first coordinate  of $\vec u$ is
\begin{align} \label{0} \nonumber
a_0 \displaystyle \sum_{j=1}^{k+1} \alpha_j^{r-1} A_{k+1,j}  & = a_0 \displaystyle \sum_{j=1}^{k+1} \alpha_j^{r-1} \Bigl( (-1)^{j+1} \prod_{1 \le s < t \le k+1 \atop s,t \neq j}(\alpha_t-\alpha_s) \Bigr)\\ 
& = \frac{a_0}{\alpha} \displaystyle \sum_{j=1}^{k+1} \Bigl(\frac{\alpha}{\alpha_j} \prod_{1 \le s < t \le k+1 \atop s,t \neq j}(\alpha_t-\alpha_s) \Bigr)  = \frac{a_0}{\alpha} A_{k+1},\\ \nonumber
\end{align}
where in the last equality we use  Lemma~\ref{relationAk} and the fact that $\alpha_j$'s are $r$-th roots of unity. For $2\le l \le r-k-1$, the $l$-th coordenate of $\vec u$ is equal to  
%$a_l \left(\displaystyle \sum_{j=1}^{k+1} \alpha_j^{n-1} A_{k+1,j} h_{l,j}(k+1)\right) $ 
\begin{align} \label{a_l} \nonumber
a_l \displaystyle \sum_{j=1}^{k+1} \alpha_j^{r-1} A_{k+1,j} h_{l,j}(\vec{\alpha}_{k+1}) & =  a_l \displaystyle \sum_{j=1}^{k+1} \alpha_j^{r-1} \Bigl((-1)^{j+1} \prod_{1 \le s \le t \le k+1 \atop s,t \neq j}(\alpha_t-\alpha_s)  \displaystyle\sum_{i=1 \atop i \neq j}^{k+1} \frac{\alpha_i^{l+k-1}}{\prod_{m=1 \atop m \neq i,j}^{k+1}(\alpha_i-\alpha_m)} \Bigr)\\ \nonumber
& = a_l \displaystyle \sum_{j=1}^{k+1} \displaystyle\sum_{i=1 \atop i \neq j}^{k+1} \bigl(\alpha_j^{r-1} \alpha_i^{l+k-1}(-1)^{j+1}(-1)^{k-i} \displaystyle\prod_{1 \le s \le t \le k+1 \atop s,t \neq i,  j}(\alpha_t-\alpha_s)\epsilon_{i,j}\bigr) \\ 
& = \frac{a_l}{\alpha} \displaystyle \sum_{j=1}^{k+1} \sum_{i=1 \atop i \neq j}^{k+1} \bigl( \frac{\alpha}{\alpha_j} \alpha_i^{l+k-1}(-1)^{k+j-i+1} \displaystyle\prod_{1 \le s \le t \le k+1 \atop s,t \neq i,  j}(\alpha_t-\alpha_s)\epsilon_{i,j}\bigr). \\ \nonumber
\end{align}
Let  $G_{k+1}$ denote the polynomial  $$G_{k+1} = \displaystyle \sum_{j=1}^{k+1} \sum_{i=1 \atop i \neq j}^{k+1} \Bigl( \frac{x_1 \cdots x_{k+1}}{x_j} x_i^{l+k-1}(-1)^{k+j-i+1} \displaystyle\prod_{1 \le s \le t \le k+1 \atop s,t \neq i,  j}(x_t-x_s)\epsilon_{i,j}\Bigr).$$
We observe that for $x_i=x_j$, $i \ne j$, we have $G_{k+1} = 0$ and therefore $(x_i - x_j)$ divides $G_{k+1}$ for all $ i \ne j$. We conclude that $A_{k+1}$ divides  $G_{k+1}$,
%For $x_i = \alpha_i, x_j= \alpha_j$, with $1 \le i < j \le k+1$, and $\alpha_i = \alpha_j$ this polynomial is zero. Then $A_{k+1}$ divides $G(x_1,\cdots, x_{k+1})$,
%$\left(\displaystyle \sum_{j=1}^{k+1} \alpha_j^{n-1} A_{k+1,j} h_{l,j}(k+1)\right) $, 
 we can write
\begin{align*}
\frac{G_{k+1} }{A_{k+1}}  & = \sum_{i=1}^{k+1} \frac{x_1 \cdots x_{k+1} x_i^{l+k-1}}{\prod_{m=1 \atop m \neq i}^{k+1} (x_i-x_m)} \sum_{j=1 \atop j \neq i}^{k+1} \frac{(x_i-x_j)}{\prod_{r=1 \atop r \neq i,j}^{k+1}(x_r-x_j)}\\
& = \sum_{i=1}^{k+1} \frac{x_i^{l+k}}{\prod_{m=1 \atop m \neq i}^{k+1} (x_i-x_m)} \sum_{j=1 \atop j \neq i}^{k+1} \frac{x_1 \cdots x_{k+1}}{x_i x_j }\frac{1}{\prod_{r=1 \atop r \neq i,j}^{k+1}(x_r-x_j)}.\\
\end{align*}
Fixing  $i$, it follows from Lemma \ref{relationAk} that $ \displaystyle\sum_{j=1 \atop j \neq i}^{k+1}\frac{x_1 \cdots x_{k+1}}{x_i x_j }\frac{1}{\prod_{r=1 \atop r \neq i,j}^{k+1}(x_r-x_j)}= 1.$ Therefore 
\begin{align} \label{l} G_{k+1} = A_{k+1}  \sum_{i=1}^{k+1} \frac{x_i^{l+k}}{\prod_{m=1 \atop m \neq i}^{k+1} (x_i-x_m)} = A_{k+1} h_l(k+1).
\end{align} 
By \eqref{a_l} we have $$a_l \displaystyle \sum_{j=1}^{k+1} \alpha_j^{r-1} A_{k+1,j} h_{l,j}(\vec{\alpha\!\!\!\!\alpha}_{k+1}) =  \frac{a_l}{\alpha}A_{k+1} h_l(k+1).$$ From \eqref{k+1}, \eqref{0} and \eqref{l}  we conclude that
\[(a_0, \dots, a_{r-k-1} ) \boldsymbol{\cdot} (1, h_1(k+1),\dots, h_{n-k-1}(k+1) = 0.\]

\end{proof}

\begin{remark} \label{rem2}
\begin{enumerate}
\item  Let  $\lambda$ be a root of $g(x)$. Multiplying $f(\lambda)$ by $\lambda^i$ we obtain $a_{r-i} + a_{r-i+1}\lambda+  \cdots + a_{r-i-1} \lambda^{r-i}=0$ and therefore Lemma \ref{innerproduct} is  true  for any shift of  the coefficients $a_0, a_1 , \dots, a_{r-1} $. 

\item In particular, Lemma \ref{innerproduct} is true if $\gcd(r,q)=1$, since in this case $g(x)$ has only simple roots.
\end{enumerate}
\end{remark}
For what follows, we will need the following definition.

\begin{definition}
Let $f(x)\in \F_q[x]$ be a monic polynomial of degree \(n\) such that \(f(0)\ne 0\). The {\em reciprocal polynomial} $f^{*}$ of the polynomial $f$ is defined by $f^{*}(x)  = \frac{1}{f(0)}x^n f\bigl(\frac1x\bigr)$. The polynomial $f$ is {\em self-reciprocal} if $f = f^{*}$. 
\end{definition}

Now,  we show how to  find the rank and an equivalent reduced matrix to the  circulant matrix $C$.

\begin{theorem}  \label{rank}
Let $A=A(a_0,a_1, \dots, a_{r-1})$ be a circulant matrix  over $\F_q$ and assume $\gcd(r,p)=1.$ If $f(x)$ is the  polynomial associated to $A$ is such that $g(x) = \gcd(f(x), x^r-1) $ is a self-reciprocal polynomial with $\deg g(x) = m$ then  rank$(A)$ is $l = r-m$ and there exists $M\in \text{GL}_r(\F_q)$ such that $MAM^T = \left(\begin{array}
{c|c}
R& 0\\ \hline
0&0
\end{array}\right),$
where $R= (r_{i,j})$ denotes the $l \times l$ matrix defined by $r_{ij} = a_{ij}$  for  $0 \le i,j \le l$ and $M^T$ is the transpose matrix of $M$. 
\end{theorem}

\begin{proof}
Let $\alpha_1, \dots, \alpha_m$ be the roots of $g(x)$. Let  $ B_i$ be  the matrices obtained from the identity matrix  changing the entries of the $r- i+1$-th row by 
\[(1,\, h_1(\alpha_1, \dots, \alpha_i) ,\,\dots, \, h_{r-i}(\alpha_1, \dots, \alpha_i) ,\, 0 ,\,\dots, \, 0).\]
Observe that  $$B_1 = \displaystyle \begin{psmallmatrix}
1&0&\cdots &0&0\\
0&1&\cdots&0&0\\
\vdots&\ddots&\cdots&\ddots&\vdots\\
0&0&\cdots&1&0\\
1&\alpha_1&\cdots&\alpha_1^{r-2}&\alpha_1^{r-1}\\
\end{psmallmatrix},$$  and  since $\alpha_1$ and $\alpha_1^{-1}$ are roots of $g$, the product $B_1AB_1^{T}$ has  last row and column with null entries.
From Lemma~\ref{innerproduct} and Remark~\ref{rem2} it follows that $MAM^T=\left(
\begin{array}{c|c}
R  & 0 \\ \hline
 0 & 0
\end{array}\right)$, where $M=B_m B_{m-1} \cdots B_2 B_1$ and $R$ is the matrix $A$ reduced to its first $l$ rows and $l$ columns.
% Let define   $M=  $, then from Lemma~\ref{innerproduct} and Remark~\ref{rem2} we have that 
%\[MA M^{T} = \begin{pmatrix}
%R & 0 \\
%0&0\\
%\end{pmatrix} ,\]
%where  $R$  is the matrix $C$ reduced to its first $l$ rows and $l$ columns. 

%{\color{red}
%Let $ q \equiv 1 \pmod 4$, we consider the extension $\F_{q^r}$ such that $4$ divide $r$.
%
% Let $\theta \in \F_{q}$ such that $\theta^2 = -1$. Then $(x-\theta)(x-\theta^{-1}) = x^2+1$ and $\gcd(x^2+1, x^r-1) = x^2+1$. The $\F_q$-linearized polynomial is $F(x) = x^{q^2} + x$.
%}
\end{proof}

\begin{example}\label{ex1}
Let $q=27$, $r=7$ and $\Omega_r$ denote the $r$-th cyclotomic polynomial. Since  $\ordem_r\, q  = 2$, $\Omega_r$ splits into three monic irreducible  polynomials over $\F_q[x]$ of  degree $2$. Let $\langle a \rangle = \F_{27}^{*} ,$ where we can choose $a$ with minimal polynomial $x^3+2x+1$.  Then 
\[\Omega_{r}(x) =(x^2+2a^2x+1)(x^2+(2a^2+a+2)x+1)(x^2+(2a^2+2a+2)x+1).  \]
Let us  define
\begin{align*} 
f(x) = & (x^2+2a^2x+1)(x^2+(2a^2+a+2)x+1)(x-a)\\
= & x^5 + x^4(a^2+2)+x^3(a^2+a+1)+x^2(2a+1)+x(a^2+2a)+2a
\end{align*}
and therefore the circulant matrix  associated to the polynomial $f(x)$ is  $$A = \begin{psmallmatrix}
2a&a^2+2a&2a+1&a^2+a+1&a^2+2&1&0\\
0&2a&a^2+2a&2a+1&a^2+a+1&a^2+2&1\\
1&0&2a&a^2+2a&2a+1&a^2+a+1&a^2+2\\
a^2+2&1&0&2a&a^2+2a&2a+1&a^2+a+1\\
a^2+a+1& a^2+2&1&0&2a&a^2+2a&2a+1\\
2a+1&a^2+a+1& a^2+2&1&0&2a&a^2+2a\\
a^2+2a&2a+1&a^2+a+1& a^2+2&1&0&2a\\
\end{psmallmatrix}.$$
Since
\begin{align*}
g(x) = \gcd(f(x), x^r-1)& =  (x^2+2a^2x+1)(x^2+(2a^2+a+2)x+1) \\&= x^4+x^3(a^2+a+2)+x^2(2a^2+a)+x(a^2+a+2)+1\end{align*}
is a self-reciprocal polynomial, it follows from Theorem \ref{rank} that rank$(A)$ is $3$ and the  reduced matrix  associated to $A$ is 
$A' = \begin{psmallmatrix}
2a&a^2+2a&2a+1\\
0&2a&a^2+2a\\
1&0&2a\\
\end{psmallmatrix}$. In addition,   $\det A'= a^4+12a^3+a = a^2 \neq 0$.
\end{example}

\section{The number of affine rational points  of $y^q-y = x  F(x) - \lambda$}

In this section, in order to find the number of affine rational points  of the curve $y^q-y = x  F(x) - \lambda$, where $F(x)$ is a $\F_q$-linearized,  we determine the number of solutions of the equation $\Tr(xF(x)) = \lambda$ in $\F_{q^r}.$ In fact, by \eqref{Npontos}
%we  determine an equivalent equation  which depends only on the variable $x$ and the trace function of the extension $\F_{q^r}$. From Remark \ref{mm} we have that  $\Tr(xF(x))$ defines a quadratic form.
% The objective is to find a matrix associated to this quadratic form and use  well known results about quadratic forms. 
%Let $g(x) = xF(x) - \lambda $ where $\lambda \in \F_q $ and $F(x)=\sum_{i=0}^{l} a_l x^{q^l}.$ 
%we recall that  $\mathcal C_{F,\lambda}$ is the curve defined by the equation 
\begin{align}  \label{curve}
 y^q-y=x  F(x)-\lambda. \end{align}
%We are interested in finding $N_r(\mathcal C_{F,\lambda})$. % the number of affine rational points   of $\mathcal C_{F,\lambda}$ in $\F_{q^r}^2$. 
%Applying $\Tr$  to both sides of \eqref{curve} we obtain 
%\begin{equation} \label{traco}
%\Tr (x F(x)) = r \lambda.
%\end{equation}
%It follows that each rational point $(x_0,y_0) \in \mathcal C_{F,\lambda}(\F_{q^r})$ yields a solution of $\Tr(x_0(F(x_0))= r\lambda$. On the other hand, for  each  solution of $\Tr(x_0(F(x_0))= r\lambda$ we have   $q$ affine rational points  of $\mathcal C_{F,\lambda}(\F_{q^r}^{*})$ of the form $(x_0,y_j)$, where the $y_j$'s are the solutions of the equation $y^q-y = x_0F(x_0) - \lambda.$ Therefore
We recall that we have
%Since $\Tr_{\F_{q^r}/\F_q}(x  F(x))$ defines a quadratic form \vm{todo mundo j\'a sabe disso},  we have that \vm{we have}
$$ N_r(\mathcal{C}_{F,\lambda}) = q S_{\lambda},$$
where  $S_{\lambda} = |\{ x\in \F_{q^r}\mid \Tr(xF(x))=r\lambda \}|$. 
%\vm{finalmente o $S_\lambda$ aparece! Essa \'e a transi\c c\~ao da contagem de um objeto para outro; \'e simples, mas \'e um passo importante no desenvolvimento do artigo e deveria aparecer com destaque, pelo menos em uma proposi\c c\~ao separada.}

In that follows,   $\mathcal P$ denotes the $r \times r $ cyclic permutation matrix  
\begin{equation}\label{matrizpermutacao}
\mathcal P = \begin{psmallmatrix}
0&1&0& \cdots & 0\\
0&0&1& \cdots &0 \\
\vdots& \cdots & \ddots & \cdots &\vdots\\
0&0&0& \cdots &1\\
1&0&0& \cdots & 0
\end{psmallmatrix}.
\end{equation}
 %and we set  $\mathcal{P}_l = \mathcal{P}^l$  for each  non-negative integer $l$. %\vm{Fica melhor 'Let $\mathcal P$ denote \ldots; we set  $\mathcal{P}_l = \mathcal{P}^l$\ldots'. Essa  nota\c c\~ao s\'o se explica por motivos tipogr\'aficos (${\mathcal P}^{l^T}$ fica feio). Seria bom  informar isso para o leitor, algo como 'For aesthetic/typographical (escolha) reasons, we write $\mathcal{P}_l = \mathcal{P}^l$'. Por outro lado, simplifica a vida esquecer tudo isso e escrever $\mathcal{P}^l+\left(\mathcal{P}^l\right)^{^{\ }_T}$ (gambiarra no \LaTeX). Voto a favor dessa \'ultima op\c c\~ao.}
 The following proposition associates the $\F_q$-linearized  polynomial $F(x)$ with an appropriated circulant matrix. 
%For  $\Gamma = \{\alpha_1, \alpha_2, \dots, \alpha_r \}$ a basis of $\F_{q^r} $ over $\F_q$ we define the matrix $M_{\Gamma}  = \begin{psmallmatrix}
%\alpha_1 & \alpha_1^q &\cdots  &\alpha_1^{q^{r-1}} \\
%\alpha_2 & \alpha_2^q & \cdots & \alpha_2^{q^{r-1}} \\
%\vdots & \cdots & \ddots & \vdots \\
%\alpha_r & \alpha^q & \cdots & \alpha_r^{q^{r-1}}
%\end{psmallmatrix}. $

\begin{proposition} \label{quadform}
Let $F(x) = \sum_{i=0}^{l} a_l x^{q^l}$ be  $\F_q$-linearized. For $\lambda \in \F_q$, the number of solutions of $\Tr(xF(x))=r \lambda$ in $\F_{q^r}$ is equal to the number of solutions \(\vec z=(z_1,z_2,\dots,z_n)^T\in \F_{q}^r\) of the quadratic form 
\[{\vec z\,}^T A \vec z= r\lambda\]
%\[\left(z_1\,\, z_2 \,\, \cdots \,\, z_r \right) A  \left(\begin{array}{c}
%z_1\\
%z_2 \\
%\vdots \\
%x_r
%
%\end{array} \right)= r \lambda, \] 
where  $A =  \frac{1}{2} \sum_{i=0}^{l} a_i(\mathcal{P}^i+\left(\mathcal{P}^i\right)^{^{\ }_T}). $ 
\end{proposition}

\begin{proof}
Let   $\Gamma = \{ \beta_1, \dots , \beta_r \}$ be a basis of $\F_{q^r}$ over $\F_q$ and  $$N_{\Gamma}  = \begin{psmallmatrix}
\beta_1 & \beta_1^q &\cdots  &\beta_1^{q^{r-1}} \\
\beta_2 & \beta_2^q & \cdots & \beta_2^{q^{r-1}} \\
\vdots & \vdots & \ddots & \vdots \\
\beta_r & \beta_r^q & \cdots & \beta_r^{q^{r-1}}
\end{psmallmatrix}. $$ Then $N_{\Gamma}$ is  an invertible  matrix and for  $x \in \F_{q^r}$ we can write  $x = \sum_{j=1}^{r} \beta_jx_i$,  where $x_1, \dots, x_r \in \F_q$. The equation  $\Tr(x  F(x)) = r\lambda$ is equivalent to \begin{equation} \label{trace}
\sum_{j=0}^{r-1} x^{q^j}   F(x)^{q^j} = r \lambda.
\end{equation}
Since $F(x)$ is a $\F_q$-linearized  and the trace  is $\fq$-linear, we need to express the monomials of the form  $x\cdot x^{q^l}$ in terms of the basis $\Gamma$. We have

% \vm{Estou confuso. 'it is enough' significa que em algum momento, mesmo implicitamente,  se disse algo do tipo 'we want to show that'; a express\~ao completa \'e do tipo 'in order to do this, it is enough to\ldots'. S\'o consigo imaginar que tem a ver com considerar mon\^omios em $\Tr(x\cdot F(x))=r \lambda$, mas n\~ao vejo como reduzir as solu\c c\~oes dessa equa\c c\~ao a parcelas monomiais.
%
%Acho (achismo mesmo!) que a ideia \'e escrever $\Tr(xF(x))=r \lambda$ com combina\c c\~ao linear dos $\tilde L_l(z)$ e  $P$ como a mesma combina\c c\~ao linear dos $\mathcal{P}_l$.  } \\ to consider  the monomials \vm{'consider monomials'} of the form  $x\cdot x^{q^l}$.
%
% Let us denote \vml{Let} $L_l(x) = \Tr(x \cdot x^{q^l})$, that can be rewritten as \vm{esse 'that can be rewritten as' n\~ao est\'a legal; 'we then have' \'e suficiente.
%
%A nota\c c\~ao $L_l(x)$ dura apenas tr\^es linhas, deixar $\Tr(x \cdot x^{q^l})$ n\~ao seria suficiente?}
\begin{align*}
\Tr(x^{q^l+1}) =\sum_{j=0}^{r-1} x^{q^j} \cdot  (x^{q^l})^{q^j}   = \sum_{j=0}^{r-1}\left( \sum_{s=1}^{r} \beta_s x_s\right)^{q^j} \left( \sum_{k=1}^{r} \beta_k x_k  \right)^{q^{j+l}}
= \sum_{s,k=1}^{r} \Biggl( \sum_{j=0}^{r-1}  \beta_s^{q^j} \beta_k^{q^{j+l}}\Biggr) x_sx_k .
\end{align*}
Consequently  $\Tr(x^{q^l+1})$ has the following symmetric representation 
\[ (x_1 \,\, x_2 \,\, \cdots \,\, x_r)  B_l  \begin{pmatrix}
x_1\\
x_2 \\
\vdots \\
x_r
\end{pmatrix}, \  \text{ where } B_l = \frac{1}{2} N_{\Gamma}(\mathcal P^l+\left(\mathcal{P}^l\right)^{^{\ }_T})N_{\Gamma}^T.\]
 Making the change of variables $(z_1 \,\, z_2\,\, \cdots\,\, z_r) = (x_1\,\, x_2\,\, \cdots\,\,x_r) N_{\Gamma}$ we get
%  an equivalent system
% \vm{N\~ao vejo em nenhum lugar um sistema com o mesmo n\'umero de solu\c c\~oes, apenas uma maneira equivalente de escrever $L_l$.}
\[ \begin{pmatrix}z_1 \,\, z_2 \,\, \cdots \,\, z_r \end{pmatrix} \left[ \frac{1}{2} (\mathcal P^l+ \left(\mathcal{P}^l\right)^{^{\ }_T}) \right] \begin{pmatrix}
z_1\\
z_2 \\
\vdots \\
z_r
\end{pmatrix}\]
which has the same number of solutions of $\Tr(x^{q^l+1})$. 
Using  equation \eqref{trace} and  the  definition of  $A $%=\frac{1}{2} \sum_{i=0}^{l} a_i (\mathcal P_i+\mathcal P_i^T) $
, the result follows. 
\end{proof}

The following theorem  is straightforward consequence  of   Theorems \ref{rank} and  \ref{sol} and Proposition \ref{quadform}.

\begin{theorem}\label{thF(x)}
 Let $F(x) = \sum_{i=0}^{l} a_i x^{q^i}$ be $\F_q$-linearized  and  $f(x) =\sum_{i=0}^{l} a_i x^{i}$ its associated polynomial. We assume $\gcd(r,p)=1$ and that $g(x) = \gcd(f(x), x^r-1)$ is a self-reciprocal polynomial of degree $m$. Let also $R$ be the matrix  defined as in  Theorem \ref{rank} and $a = \text{det } R$. Then for each $\lambda \in \F_q$, the number of affine rational points   in  $\F_{q^r}^2$ of the curve $y^q-y = x F(x) - \lambda$   is  
$$N_r(\mathcal C_{F,\lambda})= \begin{cases}
q^r -    q^{(r+m-2)/2}\chi((-1)^{(r-m)/2}a),& \text{ if $r+m$ is even and $\lambda \neq 0$;}\\
q^r +  (q-1) q^{(r+m-2)/2}\chi((-1)^{(r-m)/2}a),& \text{ if $r+m$ is even and $\lambda=0$;}\\
q^r + q^{(r+m-1)/2}\chi((-1)^{(r-m-1)/2}2 r \lambda a),& \text{ if $r+m$ is odd.}
\end{cases}
$$
%where $\omega(b)=\begin{cases} -1, & \text{if $b\ne 0$}\\
% q-1, & \text{if $b= 0.$}
%\end{cases}$
\end{theorem}

\begin{corollary}
Let $F$, $f$ and $g$  be  polynomials which satisfy the conditions of Theorem \ref{thF(x)}.  Then 
$$
|N_r(\mathcal C_{F,\lambda})-q^r|\le (q-1)q^{\frac{r+m-2}2}. $$
In addition, the upper (resp. lower) bound is attained if and only if  $r+m$ is even, $\lambda=0$ and $(-1)^{\frac{r-m}2}a$ is  (resp. not) a square in $\F_q$.
\end{corollary}

\begin{remark}
The curve 
$\mathcal C_{F,\lambda}$,
 where $F(x) = \displaystyle \sum_{i=0}^l a_i x^{q^i},a_i \in \F_q$ and $0\le l <r,$ has genus $g = \frac{(q-1)q^l}2$.
The Hasse-Weil bound for $\mathcal C_{F,\lambda}$ is given by 
\[|N_r(\mathcal C_{F,\lambda})-q^r|\le (q-1)q^{\frac{2l +r}2}. \]
Consequently, this curve is not maximal (resp. minimal) with respect to this bound. 

\end{remark}

\begin{example}
Let $q=27, r= 7$ and $f(x)$,  $g(x)$ be the polynomials of Example \ref{ex1}. The polynomial  $F(x) = x^{q^5}+(a^2+2)x^{q^4}+(a^2+a+1)x^{q^3}+(2a+1)x^{q^2}+(a^2+2a)x^q+2ax$ is the $\F$-linearized of $f(x)$.  Since $r-m$ is odd and  $\det C' = a^2$, we get from Theorem  \ref{thF(x)} 
%we have that $N_r(\mathcal C_{F,\lambda})$ is 
\[N_r(\mathcal C_{F,\lambda}) = q^7 +q^6 \chi(\lambda)=\begin{cases}
q^7+q^6, & \text{if $\lambda $ is a square in $\F_q^*$};\\
q^7-q^6, & \text{if $\lambda $ is not a square in $\F_q^*$};\\
q^7, &\text{if $\lambda=0 $.}
\end{cases}
\]
%\vm{Esse par\'agrafo fica bem melhor como 'Since $r-m$ is odd and  $\det C' = a^2$, we get from Theorem \ref{thF(x)}' e colocando $N_r(\mathcal C_{F,\lambda})=q^7 +q^6 \chi(\lambda)=\ldots$ a seguir.}

% Therefore if $\lambda $ is a square in $\F_q^*$ the number of affine rational points  is $q^7+q^6$ and if  $\lambda $ is not a square in $\F_q$ the number of affine rational points  is $q^7-q^6$. Besides that, if $\lambda =0$ the number of affine rational points  is $q^7$.
\end{example}

%{\color{red}
%
%\begin{theorem} Let $\mathcal C$ be the surface defined by the equation
%\begin{equation}\label{artin}
%\mathcal C : y^q - y =g_1(x_1)+\cdots+g_s(x_s)-\lambda
%\end{equation}
%where $g_j(x_j)=x_j P_j(x_j)$, e $P_j(x_j)$ is a $\F_q$-linearized polynomial for each $j=1,2,\dots,s$. 
%The number of $\F_{q^r}$ rational points  of $\mathcal C$ is:
%
%
%\end{theorem}
%}
In the following section we use Theorem  \ref{sol} to compute $D$ for some some special polynomials  $F(x)$.

\section{The number of affine rational points  of $y^q-y = x \cdot(x^{q^i}-x)-\lambda $} \label{sec4}
Throughout this section, for any prime $t$  and a positive integer $n$, we denote by $\left(\frac{n}{t} \right)$  the Legendre symbol  and by $\nu_t(n)$ the $t$-adic valuation of $n$, i.e., the maximum power of $t$ that divides $n$.  The objective of this section is to find an expression  for  the number of affine rational points   of the curve $y^q-y = x^{q^i+1}-x^2 - \lambda$ in $\F_{q^r}^2$ in terms  of  valuation functions  and Legendre symbols.

In the previous section, we used the fact that the number $N_r(\mathcal C_{F,\lambda})$, for $F(x)$   a $\F_q$-linearized polynomial and $\lambda \in \F_q$, is $q$ times  the number of elements $x \in \F_{q^r}$ such that $\Tr(xF(x))= r \lambda $.  

In  order to determine the number of solutions of $\Tr(xF(x)) = r \lambda$, it is necessary to establish the dimension of the symmetric bilinear form associated to this quadratic form, which is the content of the next proposition.

%\vm{Sugeri antes que todos os resultados gerais sobre formas quadr\'aticas fossem concentrados em um \'unico lugar. Parece-me ser esse o caso tamb\'em para os resultados sobre $\Tr(x \cdot F(x))$.}

\begin{proposition} \label{dimgeral}
Let $0<i<r$ be integers and $F(x) = \sum_{j=0}^{i} a_jx^{q^j}$ a $\F_q$-linearized. Let  $\Phi_F(x) = \Tr(x F(x))$.
% be a quadratic form over $\F_{q^r}$
% \vm{$\Phi_i(x)$ n\~ao \'e 'a' (tradu\c c\~ao: uma) quadratic form, pois \'e determinada por $F(x)$. E tamb\'em todo mundo j\'a sabe que $\Tr(x \cdot F(x))$ \'e uma forma quadr\'atica\ldots Chatice minha: acho estranha a nota\c c\~ao $\Phi_i$, que \'e associada apenas ao grau de $F(x)$; n\~ao seria melhor $\Phi_F$, assim como $\varphi_F$?} and $\varphi_F(x,y)$ the symmetric bilinear form associated to $\Phi_F$. 
If $a_0 \neq 0$, then
\begin{equation} \label{dimP}
\dim \text{rad}(\varphi_F) = \deg  \Bigl(\gcd\Bigl(\sum_{j=0}^{i}a_j (x^{j} +x^{r-j}), x^{r}-1 \Bigr)\Bigr).
\end{equation}
\end{proposition}

\begin{proof}
In order to  determine the dimension of the radical of $\varphi_F$ it is sufficient  to compute the dimension of the radical
%l of the symmetric bilinear form $\varphi_F(x,y)$, i.e., 
%\vm{j\'a comentei essa defini\c c\~ao anteriormente; aqui noto apenas que n\~ao \'e necess\'ario repet\'\i-la}
$$\dim_{\fq}~\{x \in \F_{q^r}| \varphi_F(x,y) = 0 \text{ for all } y \in \F_{q^r}\}.$$ 
In fact %calculating $\varphi(x,y)$
\begin{align}\label{q1}
\nonumber
\varphi(x,y)&  = \Tr \Bigl(  \sum_{j=0}^{i} a_j(x+y)^{q^j+1} -\sum_{j=0}^{i} a_jx^{q^j+1} -  \sum_{j=0}^{i} a_jy^{q^j+1}\Bigr) \\ \nonumber
& = \sum_{l=0}^{r-1}  \bigl( \sum_{j=0}^{i} a_j(x+y)^{q^{j+l}+q^l} - \sum_{j=0}^{i} a_jx^{q^{j+l}+q^l} -  \sum_{j=0}^{i} a_jy^{q^{j+l}+q^l }\bigr) \\ \nonumber
& = \sum_{j=0}^{i} a_j   \bigl( \sum_{l=0}^{r-1}  x^{q^{j+l}}y^{q^l} +x^{q^l}y^{q^{j+l}}\bigr) \\
& = \sum_{j=0}^{i} a_j \bigl(\sum_{j=0}^{r-1} ((x^{q^j}+x^{q^{r-j}})y)^{q^l} \bigr)\nonumber\\ 
& = \sum_{j=0}^{i} a_j \Tr((x^{q^j}+x^{q^{r-j}})y) = \Tr\bigl( \sum_{j=0}^{i} a_j(x^{q^j}+x^{q^{r-j}})y \bigr).
\end{align} 
 It follows that $\varphi_F(x,y) =0$ for all $y \in \F_{q^r}$ is equivalent to
\begin{align}\label{d1}
 \sum_{j=0}^{i} a_j(x^{q^j}+x^{q^{r-j}}) = 0.
\end{align}
The  $\F_q$-linear subspace of $\F_{q^r}$ determined by  \eqref{d1} is the set of roots of $$g(x) =\gcd(H(x), x^{q^{r}}-x  ), \text{ where } H(x)= \sum_{j=0}^{i} a_j(x^{q^j}+x^{q^{r-j}}). $$ Since $g$ is a $\F_q$-linearized, the degree of the associated polynomial  gives us the dimension of  radical of $\varphi_i$, which is  the degree of   $\gcd(x^{r} - 1,  \sum_{j=0}^{i} a_j(x^{^j}+x^{{r-j}}))$. This finishes the proof. \end{proof} 

\subsection{ The special case  $F(x) = x^{q^i} -x$.}
For this case we  explicitly determine the dimension of the quadratic form $\Tr(xF(x)).$ Besides that, we will use this information to compute $N_r(\mathcal C_{F,\lambda})$, that is given in Theorem \ref{th2}.  In order to simplify the notations, we define the following quadratic form.
\begin{definition}
Let $i,r$ be integers such that $0< i<r$. We define 
\begin{align*} \Phi_i: \F_{q^r} &\to \fq\\
 x &  \mapsto \Tr(x^{q^i+1}-x^2).\\ 
\end{align*}
\end{definition}

The following corollary is consequence of Proposition \ref{dimgeral}.
%{\color{red}
%\begin{corollary}\label{dimensao}
%Let  $\Phi_i: \mathbb{F}_{q^r} \to \fq$ given by $\Tr(x^{q^i+1}-x^2)$ , where  $0<i < r$, and $\varphi_i(x,y) $  its associated symmetric bilinear form.  Let $r=p^u \tilde{r} $ and  $i=p^s \tilde{i}$, where $u,s$ are non-negative  integers such that  $\gcd(p,\tilde{r} )= \gcd(p, \tilde{i})=1$. Then
%\begin{equation} \label{dim}
%\dim \ker(\varphi_i) = \gcd(\tilde{r}, \tilde{i}) \min(p^u,2p^s).
%\end{equation}
%\end{corollary}}

\begin{corollary}\label{dimensao}
Let $i,r$ be integers such that $0<i<r$ and   $\varphi_i(x,y) $  the associated symmetric bilinear form of $\Phi_i$.  Let $r=p^u \tilde{r} $ and  $i=p^s \tilde{i}$, where $u,s$ are non-negative  integers such that  $\gcd(p,\tilde{r} )= \gcd(p, \tilde{i})=1$. Then
\begin{equation} \label{dim}
\dim \text{rad }(\varphi_i) = \gcd(\tilde{r}, \tilde{i}) \min(p^u,2p^s).
\end{equation}
\end{corollary}

\begin{proof}
By Proposition \ref{dimgeral} it is enough to find the dimension of the linear  space determined by the roots of
\begin{align} \label{z1}
H(x)=\gcd(x^{q^i} + x^{q^{r-i}} -2x, x^{q^r}-x).% = \gcd(x^{2q^i}-2x^{q^i} +x, x^{q^r}-x). 
\end{align}
Since $r=p^u \tilde{r} $, $i=p^s \tilde{i}$ , the associated polynomial to the $\F_q$-linearized polynomial  $H(x)$ % Equation \eqref{z1} 
is 
\begin{align*}
h(x)&= \gcd \left(x^{i} -2+x^{r-i} , x^r-1 \right) =\gcd \left(x^{2i} -2x^i+x^{r} , x^r-1 \right) \\
&= \gcd((x^{\tilde{i}}-1)^{2p^s},(x^{\tilde{r}} - 1)^{p^u})= (x^{\gcd(\tilde{r},\tilde{i})}-1)^{\min{(p^u,2p^s)}}.
\end{align*}
Since the degree of  $h(x)$ is equals to the dimension of the radical, we conclude that $\dim \text{ rad }(\varphi_i) = \gcd(\tilde{r}, \tilde{i}) \min(p^u,2p^s).$
\end{proof}

Using Theorem \ref{sol} and  the  previous corollary we can determine the number of solutions of $\Tr(x^{q^i+1} -x^2)=r\lambda$ in $\F_{q^r}$,  which will give us a complete description of  
$N_r(\mathcal C_{F,\lambda})$.
%the number of rationals points in $\F_{q^r}$ of the curve $y^q-y = x^{q^i+1} -x^2-\lambda$.
\begin{lemma}\label{rimpar}
Let $i,r $ be integers such that   $0< i < r$ and $\gcd(r,2p)=1$. Let  $v$ be the dimension of the radical of the  associated bilinear symmetric form $\Phi_i$. 
%{\color{red}Let $i,r $ be integers such that   $0< i < r$ and $\gcd(r,2p)=1$. Let $\Phi_i(x) :  \mathbb{F}_{q^r} \to \fq$ given by $\Tr(x^{q^i+1}-x^2)$ and $v$ the dimension of the radical of the  associated bilinear symmetric form $\varphi_i$.}
 Let  $i=p^s \tilde{i}$, where  $s$ is a non-negative integer and  $\gcd(\tilde{i},p) =1$. Then $r+v$ is even and,  for  $\lambda \in \F_{q}^*$, the constant $D$ of Theorem  \ref{sol} it is given by 
\[ D = \prod_{j=1}^{u} \left(\frac{q}{p_j}\right)^{\max\{0, \nu_{p_j}(r)-\nu_{p_j}(i)\}}.\]
where $r = p_1^{a_1}\cdots p_u^{a_u}$  is the prime factorization of $r$.
\end{lemma}

\begin{proof}
Let $M_\lambda$ be the set of solutions  of $\Tr(x^{q^i+1}-x^2) = \lambda$ in $\F_{q^r}$. Then $S_\lambda=|M_\lambda|$ is given by Equation  \eqref{casopar}. For each $\lambda \in \F_q^*$, if $\Tr(x^{q^i+1}-x^2) = \lambda$ we have
 \begin{align}\label{outrasol}
\Tr((x^{q^j})^{q^i+1} - (x^{q^j})^2) = \Tr((x^{q^i+1}-x^2)^{q^{j}}) = \Tr(x^{q^i+1}-x^2) = \lambda,
\end{align}
for all $0 \le j \le r-1.$   

 We first consider the case $r= p_1^a$, where $p_1$ is an odd prime and $\gcd(p_1,p) =1$.  By Equation \eqref{outrasol},  for each   $\alpha\in M_{\lambda}$ we can associate   another $d-1$ elements of $M_{\lambda}$, where  $d$ is the smallest positive divisor of $r = p_1^a$ such that  $\alpha^{q^{d}} = \alpha$. For each $\alpha \in M_{\lambda}$ we have $d>1$, because otherwise  $\alpha^q = \alpha$, which implies 
\(\alpha^{q^i+1}-\alpha^2 = \alpha^2-\alpha^2 = 0,\)
which is a contradiction because $\lambda\ne 0$. Then $d$ is a multiple of $p_1$ and Equation   \eqref{casopar} of Theorem \ref{sol} can be rewritten modulo  $p_1$ as
\[q^{r-1} - D q^{(r+v-2)/2} \equiv 0 \pmod {p_1},\]
which is equivalent to 
\[D \equiv (q^{(r+v-2)/2})^{-1} \equiv q^{(r+v-2)/2}\pmod {p_1},\]
where in the last congruence we use the fact that $D = \pm 1$. By Lemma \ref{dimensao} we obtain
\begin{align*}
D &\equiv q^{({p_1}^a+ {p_1}^{\min\{a,\nu_{p_1}(i)\}})/2-1} \pmod {p_1}\\
& \equiv q^{p_1^{\min{(a,\nu_{p_1}(i)})}(p_1^{(a-\min{(a,\nu_{p_1}(i))}}+1)/2-1 } \pmod {p_1}\\
& \equiv q^{(p_1^{(a-\min(a,\nu_{p_1}(i))}-1)/2} \pmod {p_1}\\
&\equiv  q^{(p_1^{\max\{0,a-\nu_{p_1}(i)\}}-1)/2} \pmod {p_1}\\
&\equiv  \left(\frac{q}{p_1}\right)^{(p_1^{\max\{0,a-\nu_{p_1}(i)\}}-1)/(p_1-1)} \pmod {p_1}\\
\end{align*}
Since $\left(\frac{q}{p_1}\right)$ assumes only the values  $\{-1,1\}$ and $\frac{p_1^l-1}{p_1-1} \equiv l \pmod 2$, we  conclude  that
\[D =  \left(\frac{q}{p_1}\right)^{\max\{0,a-\nu_{p_1}(i)\}}.\]

 Now we consider the general case  $r=p_1^{a_1}\cdots p_u^{a_u}$,  with $u\ge 1$. %with $t_j$ being distinct odd primes satisfying $\gcd(t_i,p)=1$. 
We will prove the result by induction on $u$. We already proved the case when $u=1$. Now suppose $u\geq 2$.  It follows from  Lemma \ref{dimensao} that the dimension of the radical of the bilinear symmetric form associated to $\Phi_i(x)$ is $v = \gcd(p_1^{a_1}\cdots p_u^{a_u},i)$ and therefore $v$ divides $p_1^{a_1}\cdots p_u^{a_u}$ and $r+v$ is even. Using Theorem \ref{sol}, for $\lambda \in \F_q^{*}$, we obtain that
\[S_{\lambda} = q^{r-1} - Dq^{(r+v-2)/2}.\]

Now let $\lambda\in \F_q^*$ and $r=\tilde{r} p_u^{a_u}$ where $\tilde{r}=p_1^{a_1}\cdots p_{u-1}^{a_{u-1}}$. We now consider the subfield  $\F_{q^{\tilde{r}}}\subset \F_{q^r}$.   The induction hypothesis, the number of solutions of $\Tr_{\F_{q^{\tilde{r}}}/\F_{q}}(x^{q^i+1}-x^2) = \lambda$ is 
\[S_{\lambda,\tilde{r}} =  q^{\tilde{r}-1}- \prod_{j=1}^{u-1} \left( \frac{q}{p_j} \right)^{\max\{0,\nu_{p_j}(r)-\nu_{p_j}(i)\}} q^{(\tilde{r}+v_1-2)/2},\]
 where  $v_1$ is the dimension of radical of the bilinear symmetric form associated to  $\Tr_{\F_{q^{\tilde{r}}}/\F_{q}}(x^{q^i+1}-x^2)$. From  Lemma~\ref{dimensao} we know that $v_1 = \gcd(\tilde{r},i)$ which implies $v = v_1 \cdot \gcd(p_u^{a_u},i)$.
Since $\displaystyle\F_{q^{\tilde{r}}} \cap \F_{q^{p_u^{a_u}}} = \F_q $, the solutions which are not in $\F_{q^{\tilde{r}}}$ can be grouped in sets of size congruent to $0$ $\pmod{p_u}$. In fact, since $\alpha$ is a solution then $\alpha^{q^j}$ is also a solution   and $\alpha \in \F_{q^r}$, it follows that  there exists $d>1$  dividing $p_u^{a_u}$ such that $\alpha^{q^d} = \alpha$. Then
\[S_{\lambda} \equiv S_{\lambda,\tilde r} \pmod {p_u},\]
which is equivalent to 
\[ q^{\tilde{r}p_u^{a_u}-1} - Dq^{(\tilde{r}p_u^{a_u}+v-2)/2} \equiv q^{\tilde{r}-1}-\prod_{j=1}^{u-1} \left( \frac{q}{p_j} \right)^{\max\{0,\nu_{p_j}(r)-\nu_{p_j}(i)\}} q^{(\tilde{r}+v_1-2)/2}\pmod {p_u}.\]
Since $q^{p_u} \equiv q \pmod {p_u}$, the previous equation is equivalent to \begin{align}\label{Drimpar}
 Dq^{(\tilde{r}p_u^{a_u}+v-2)/2} \equiv  \prod_{j=1}^{u-1} \left( \frac{q}{p_j} \right)^{\max\{0,\nu_{p_j}(r)-\nu_{p_j}(i)\}} q^{(\tilde{r}+v_1-2)/2}\pmod {p_u}.
\end{align}
Now let $v_2 = \gcd(p_u^{a_u},i)$. We observe that $q^{(p_u^{a_u}-1)/2} \equiv \left(\frac{q}{p_u}\right)^{a_u} \pmod {p_u}$ from which it follows that
\begin{align}\label{equiv1} \nonumber
q^{(\tilde{r} +v_1-\tilde{r}p_u^{a_u}-v )/2} & \equiv q^{-\tilde{r}\left(\frac{p_u^{a_u}-1}{2}\right)}q^{\left(\frac{v_1-v}{2}\right)} \pmod {p_u}\\ \nonumber
& \equiv \left(\frac{q}{p_u}\right)^{a_u\tilde{r}} q^{\left(\frac{v_1-v}{2}\right)} \pmod {p_u}\\ \nonumber
& \equiv \left(\frac{q}{p_u}\right)^{a_u\tilde{r}} q^{\frac{v_1(1-v_2)}{2}} \pmod{p_u}\\
& \equiv \left(\frac{q}{p_u}\right)^{a_u} q^{-v_1\frac{(p_u^{\min\{a_u,\nu_{p_u}(i)\}}-1)}{2}} \pmod{p_u}.\\ \nonumber
\end{align}
Equations \eqref{Drimpar} and \eqref{equiv1} allow us to conclude that 
\begin{align*}
D &\equiv \prod_{j=1}^{u-1} \left( \frac{q}{p_j} \right)^{\max\{0,\nu_{p_j}(r)-\nu_{p_j}(i)\}} q^{(\tilde{r}+v_1-\tilde{r}p_u^{a_u}-v)/2} \pmod {p_u}\\
& \equiv \prod_{j=1}^{u-1} \left( \frac{q}{p_j} \right)^{\max\{0,\nu_{p_j}(r)-\nu_{p_j}(i)\}} \left(\frac{q}{p_u}\right)^{a_u} q^{-v_1\frac{(p_u^{\min\{a_u,\nu_{p_u}(i)\}}-1)}{2}}\pmod{p_u}\\
& \equiv \prod_{j=1}^{u-1} \left( \frac{q}{p_j} \right)^{\max\{0,\nu_{p_j}(r)-\nu_{p_j}(i)\}} \left(\frac{q}{p_u}\right)^{a_u} \left(\frac{q}{p_u}\right)^{-\min\{a_u,\nu_{p_u}(i)\}}\pmod {p_u}\\
&\equiv \prod_{j=1}^{u} \left( \frac{q}{p_j} \right)^{\max\{0,\nu_{p_j}(r)-\nu_{p_j}(i)\}} \pmod{p_u} 
\end{align*}
and consequently $D = \displaystyle \prod_{j=1}^{u} \left( \frac{q}{p_j} \right)^{\max\{0,\nu_{p_j}(r)-\nu_{p_j}(i)\}}.$

\end{proof}
\begin{remark}
From Theorem \ref{sol} we have that  $D$  independent of the value of  $\lambda \in \F_q$.
Then, by Lemma \ref{rimpar},  for $\lambda = 0$ we have   that the value of $D$  in Equation \eqref{casopar}  is $ \prod_{j=1}^{u} \left(\frac{q}{p_j}\right)^{\max\{0, \nu_{p_j}(r)-\nu_{p_j}(i)\}}.$
\end{remark}
%The following definition will be useful in the next results for the calculation of $S_{\lambda}$.

For extensions of  degree  power of $2$ we have the following result. 
%\vm{Na verdade, n\~ao \'e s\'o esse resultado. Seria o caso de uma se\c c\~ao separada ou ent\~ao deixar claro quais resultados assumem essa hip\'otese,}
\begin{lemma}\label{D2}
Let $b,i,r$ be  integers such that $0< i < r$ and  $r = 2^b$. Let $v$ be the dimension of the radical of the associate bilinear symmetric form of $\Phi_i$. 
%{\color{red}Let $\Phi_i(x) : \mathbb{F}_{q^r}\to \fq$ given by  $\Tr(x^{q^i+1}-x^2)$ and  $v$ be the dimension of the radical of its associate bilinear symmetric form.}
 For any $\lambda \in \F_q$ the value of  $D$ as defined in Theorem \ref{sol}, is given by
\[D = \begin{cases} 
 \chi(-\lambda)& \text{ if } b=1; \\
(-1)^{(q-1)(2^b-v)/4} & \text{ if } b\ge 2 \text{ and } r+v \text{ is even} ; \\
(-1)^{(q+1)/2} \cdot \chi\left(-\frac{\lambda}{2^{b-1}}\right)& \text{ if } b\ge 2 \text{ and } r+v \text{ is odd}. \\
\end{cases} \]
%\vm{Qual \'e o pacote de \LaTeX\ que permite usar ' $\backslash $begin\{cases\}'?}
%\begin{enumerate}[(i)]
%\item If $b=1$, then $D =  \chi(-\lambda);$
%
%\item  If $b \ge 2$ and  $r+v$ is even, then $D= (-1)^{(q-1)(2^b-v)/4}$.
%
%\item If $b\ge2$ and  $r+v$ is odd, then
%$D= 
%(-1)^{(q+1)/2} \cdot \chi\left(-\frac{\lambda}{2^{b-1}}\right)\varepsilon_{\lambda}$.
%\end{enumerate}

\end{lemma} 
\begin{proof}
When  $r=2$ it  follows that  $i=1$ %. For  $\lambda \in \F_q$, we need to find the number of solutions of 
and $\Tr(\alpha^{q+1}-\alpha^2) = \lambda$ is equivalent to 
\begin{align} \label{eqcharc2}
\alpha^{q^2+q}-\alpha^{2q} +\alpha^{q+1}-\alpha^2 = \lambda,
\end{align}
which can be written as $(\alpha^q-\alpha)^2 = -\lambda$.
If $\lambda = 0$ that relation is equivalent to $\alpha^q-\alpha = 0$, and then $\alpha \in \F_q$, in which case we have $q$ solutions. For $\lambda \in \F_q^{*},$  let us consider the  following maps:
\begin{align*}
\psi : \F_{q^2}  &\to  \F_{q^2}& \text{ and }&&  \tau :\F_{q^2} & \to \F_{q^2}\\
 x &\mapsto x^q-x. && & x & \mapsto x^2 .
\end{align*}
In order to determine the number of solutions of Equation~\eqref{eqcharc2} it is enough to fix  $a\in\F_{q^2}$ such that $\tau(\psi(a)) = -\lambda$. Let $\{1,\alpha\}$ be a basis of $\F_{q^2}/\F_q$. The image of $\{1,\alpha\}$ by $\psi$ is $\{0,\beta\}$, where $\beta = \alpha^q-\alpha$. Since $\ker(\psi) = \F_q$, the image of  $\psi$ is generated by $\beta$. Therefore it  is sufficient to consider the elements of the form $c\alpha$, with $c \in \F_q$, i.e., 
\[\tau(\psi(c \alpha)) = \tau(c \beta) = c^2 \beta^2.\]
 We now claim that $\beta \notin \F_q$. For that, suppose by contradiction that $\beta^{q} = \beta$. Then
\[\alpha^{q^2}-\alpha^q-\alpha^q+\alpha = 0\]
which only happens if $-2(\alpha^q-\alpha) = 0$. But this is not possible because  $\alpha \in \F_{q^2} \setminus \F_q$ and $p \neq 2$.
Consequently   $\tau(\psi(c\alpha)) = -\lambda$ if and only if  $c^2 \beta^2 = -\lambda$, and since $\beta \notin \F_q$, this equation  has solutions if and only if  $-\lambda$ is not a square in $\F_q$. In this case,  Equation~\eqref{eqcharc2} has $2q$ solutions.

Now we consider the case when $r= 2^b$ with $b>1$. Let  $i = p^s\tilde{i}$, where $\gcd(p,\tilde{i})=1$. From  Lemma~\ref{dimensao} we know that $v= \gcd(2^b, \tilde{i})\min{(1,2p^s)} = \gcd(2^b,i)$, which implies  that $v$ is of the form $2^c$ with $0 \le  c \le b$. Let $M_\lambda$ be the set of solutions  of $\Tr(x^{q^i+1}-x^2) = \lambda$ in $\F_{q^r}$. We now consider two cases.
\begin{enumerate}
\item  $r+v$ is even.

In this case $v$ and  $i$ are even. The number of solutions of $\Tr(x^{q^i+1}-x^2) = \lambda$ is given by  Equation~\eqref{casopar}. If  $\Tr(\alpha^{q^i+1}-\alpha^2) = \lambda$ for some $\alpha \in \F_q^*$, we have 
\begin{align}\label{sols}
\Tr((\alpha^{q^j})^{q^i+1}-(\alpha^{q^j})^2) = \Tr((\alpha^{q^i+1}-\alpha^2)^{q^{j}}) = \Tr(\alpha^{q^i+1}-\alpha^2) = \lambda,
\end{align}
for  each $0 \le j \le r-1$. Since $r= 2^b\ge4$ and by Equation~\eqref{sols}, for each $\alpha\in M_{\lambda}$ we can associate  another $d-1$ elements  of $M_{\lambda}$, 
where $d$ is the smallest positive divisor of  $r = 2^b$ such that  $\alpha^{q^d}= \alpha$.  We claim that $d>2$. In fact, if  $d=1$ we have   $(\alpha^{q^i+1}-\alpha^2)^q = \alpha^2-\alpha^2 = 0$, and this implies  $\lambda = 0$, which is a contradiction.
 In the case $d =2$, for  $\alpha \in \F_{q^2}$ we have $(\alpha^{q^i+1}-\alpha^2)^{q^2} = \alpha^{q^i+1}-\alpha^2 = \alpha^{2} - \alpha^2=0$, since $i$ is even. This relation also implies $\lambda=0$, which is also a contradiction. Then Equation~\eqref{sols}  does not have solutions in these cases.  Consequently $d>2$, which implies that $4$ divides $d$ for any $\alpha \in M_{\lambda}$ and seen  Equation~\eqref{casopar} of  Theorem \ref{sol}, we obtain the  relation 
\[q^{2^b-1} - D q^{(2^b+v-2)/2} \equiv 0 \pmod 4,\]
which is equivalent to 
\begin{align*}
D& \equiv  q^{2^b-1-(2^b+v-2)/2}  \equiv q^{(2^b-v)/2}\pmod 4.\\
\end{align*}
We conclude that  $D = (-1)^{(q-1)(2^b-v)/4}$.
% and we conclude that
%\[S_{\lambda} = q^{2^b-1} -(-1)^{(q-1)(2^b-v)/4} q^{(2^b+v-2)/2} .\]

\item $r+v$ is odd.

In this case $v$ is a odd divisor of $2^b$ and therefore  $v=1$. By the same argument used in the previous case the number of solutions of $\Tr(\alpha^{q^i+1}-\alpha^2) = \lambda$  is given by Equation~\eqref{casoimpar}. Furthermore, it follows that for $\lambda \in \F_q^*$ we have $\Tr(\alpha^{q^i+1}-\alpha^2) = \lambda$ if and only if 
\begin{align}\label{sols1}
\Tr((\alpha^{q^j})^{q^i+1}-(\alpha^{q^j})^2) = \Tr((\alpha^{q^i+1}-\alpha^2)^{q^{j}}) = \Tr(\alpha^{q^i+1}-\alpha^2) = \lambda,
\end{align}
for all  $0 \le j \le r-1.$ % Since $r= 2^b\ge4$ and using Equation~\eqref{sols1}, f
Therefore for each $\alpha\in M_{\lambda}$ we can associate  another $d-1$ elements of $S_{\lambda}$, where $d$ is the smallest divisor of  $r = 2^b\ge 4$  such that $\alpha^{q^{d}} = \alpha$. The case $d=1$ does not happen, otherwise we would have $\lambda =0$.

Suppose now that $\alpha \in \F_{q^2}\cap M_{\lambda} \subset \F_{q^r}$. We then have 
%For $\alpha \in \F_{q^2}$ we can written  $(\alpha^{q^i+1}-\alpha^2)^q = \alpha^{q^{i+1}+q}-\alpha^{2q}$, and besides that  $(\alpha^{q+1}-\alpha^2)^{q^2} = \alpha^{q+1}-\alpha^2$. We are interesting in the solutions of 
\begin{align}\label{ab}
\lambda &= \Tr(\alpha^{q^i+1}-\alpha^2) = 2^{b-1} \cdot( (\alpha^{q^{i+1}+q}-\alpha^{2q}) +\alpha^{q^i+1}-\alpha^2 ).
\end{align}
As in the previous case, Equation~\eqref{ab} does not have solution in $\F_q$. Therefore  $\alpha \in \F_{q^2}\setminus \F_q$ and  the equation 
\[\lambda  = 2^{b-1} \cdot( \alpha^{q+1}+\alpha^{q+1}-\alpha^{2q} -\alpha^2 )\]
  can be written as $(\alpha^q-\alpha)^2 = \gamma,$ where  $\gamma = -\frac{\lambda}{2^{b-1}}$. Using the same argument of  item (i) of Lemma \ref{D2} it follows that $(\alpha^q-\alpha)^2 = \gamma$ has solutions in $\F_{q^2}$ if and only if $\gamma$ is not a square in $\F_q$,
which is equivalent to $-\frac{\lambda}{2^{b-1}}$ not being a square in $\F_q$, and in this case we have $2q$ solutions for Equation~\eqref{eqcharc2} in $\F_{q^2}$. 
Consequently the number of solutions of Equation  \eqref{ab} in $\F_{q^2}$ is $\left(1-\chi(-\frac{\lambda}{2^{b-1}})\right)q$ and then 
\[S_{\lambda} - \left(1-\chi\left(-\frac{\lambda}{2^{b-1}}\right)\right)q \equiv 0 \pmod 4.\]
By Theorem \ref{sol}, it follows that 
\[q^{2^b-1} + D q^{(2^b+v-1)/2} \equiv \left(1-\chi\left(-\frac{\lambda}{2^{b-1}}\right)\right)q  \pmod 4,\]
%that is equivalent to \vml{
i.e.,
$$D\equiv \left(1-\chi\left(-\frac{\lambda}{2^{b-1}}\right)\right)q^{1-(2^b+v-1)/2} - q^{(2^b-v-1)/2}\pmod 4.$$
 Therefore 
\begin{align*}D  &\equiv  \left(\left(1-\chi\left(-\frac{\lambda}{2^{b-1}}\right)\right)q^{2-2^b}-1\right)q^{(2^b-v-1)/2} \pmod 4.
\end{align*}
Since $v=1$ and   $q^2 \equiv 1 \pmod 4$ we conclude that 
\begin{align*} D & \equiv  -q^{2^{b-1}-1} \cdot \chi\left(-\frac{\lambda}{2^{b-1}}\right)  \equiv -q \cdot \chi\left(-\frac{\lambda}{2^{b-1}}\right) \pmod 4\\
\end{align*}
\end{enumerate}
and consequently \[ D = (-1)^{(q+1)/2} \cdot \chi\left(-\frac{\lambda}{2^{b-1}}\right).\]

The case $\lambda = 0$ follows from Theorem \ref{sol}, which tells us that $D$ is the same for any $\lambda\in \F_q$, if $r+v$ is even. In the case when $r+v$ is odd, we have $D=0$.
%And as consequence 
%\[S_{\lambda} = q^{2^b-1} + (-1)^{(q-1)/2} \cdot \chi\left(-\frac{\lambda}{2^{b-1}}\right) q^{(2^b+v-2)/2}. \]
\end{proof}
The following definitions are helpful to allows us to rewrite the expressions of  $D$  in a more simpler  way. 
\begin{definition}
For each  $\lambda \in \F_q$ we define
\[\varepsilon_{\lambda} = \begin{cases}
q-1 & \text{ if } \lambda =0 \\
-1 & \text{ otherwise.} \\
\end{cases} \quad \text{ and} \quad \varepsilon'_{\lambda} = \begin{cases}
0 & \text{ if } \lambda =0 \\
-1 & \text{ otherwise.} \\
\end{cases}\]
\end{definition}
In the following theorem we use Theorem \ref{sol} and Lemma \ref{D2}, to determine %in the following theorem 
the value of $S_{\lambda}$. 
\begin{theorem}\label{char2}
Let $b,i,r$ be integers such that $0< i < r$ and  $r = 2^b$. %Let $\Phi_i : \F_{q^r} \to \fq$ given by  $\Tr(x^{q^i+1}-x^2)$. % and  $v$ be the dimension of the bilinear symmetric form associated to $\Phi_i$.  
For $\lambda \in \F_{q}$, the number of solutions $S_{\lambda}$ of $\Phi_i(x) = \lambda$ in $\F_{q^r}$ is given by 
\[S_{\lambda} = \begin{cases}
 (1-\chi(-\lambda))q & \text{ if } b=1;\\
q^{2^b-1} 
+(-1)^{(q-1)(2^b-v)/4} q^{(2^b+v-2)/2}\varepsilon_{\lambda} & \text{ if } b\ge 2 \text{ and } r+v \text{ is even};\\
q^{2^b-1} + (-1)^{(q-1)/2} \cdot \chi\left(-\frac{\lambda}{2^{b-1}}\right) q^{(2^b+v-1)/2} \varepsilon'_{\lambda}& \text{ if } b\ge 2 \text{ and } r+v \text{ is odd},\\
\end{cases}\]
where  $v=\gcd(2^b, i)$ is the dimension of the radical of the bilinear symmetric form associated to $\Phi_i$.
%\begin{enumerate}[(i)]
%\item If  $b=1$:  $S_{\lambda} = (1-\chi(-\lambda))q;$
%\item If $b \ge 2$ and $r+v$ is even : $S_{\lambda} =q^{2^b-1} -(-1)^{(q-1)(2^b-v)/4} q^{(2^b+v-2)/2}$;
%\item If $b \ge 2$ and $r+v$ is odd: $S_{\lambda} = q^{2^b-1} - (-1)^{(q-1)/2} \cdot \chi\left(-\frac{\lambda}{2^{b-1}}\right) q^{(2^b+v-2)/2}\varepsilon_{\lambda}$. 
%\end{enumerate}

\end{theorem}

The results obtained in Lemmas \ref{rimpar} and \ref{char2}
can be used inductively to obtain the following result for extensions of degree $r$ satisfying  $\gcd(r,p)=1$. 

\begin{theorem} \label{anyr}
Let $b,i,r$ be  integers such that $0< i < r$, $r = 2^b\tilde{r},$ $\tilde r= p_1^{a_1}\cdots p_u^{a_u}$  the prime factorization of $\tilde r$  and $\gcd(\tilde r,2p)=1$. 
%Let $\Phi_i:\F_{q^r} \to \fq$ given by $\Tr(x^{q^i+1}-x^2)$. % and $v$ be the dimension of the bilinear symmetric form associated to $\Phi_i$.  
For $\lambda \in \F_{q}$, the number of solutions of  $\Phi_i(x)=\lambda$ in $\F_{q^r}$ is  
\[S_{\lambda} = \begin{cases}
q^{r -1} + \displaystyle \prod_{j=1}^{u} \left(\frac{q}{p_j}\right)^{\max\{0,\nu_{p_j}(\tilde{r})-\nu_{p_j}(i)\}}q^{\frac{r + 2 v_0 -2}{2}}\varepsilon_{\lambda} & \text{ if $i$ is even and } b =1;\\
q^{r-1} +\displaystyle(-1)^{(q^{\tilde{r}}-1)(2^b-v_1)/4}q^{\frac{(r-\tilde{r}v_1-2)}{2}}\varepsilon_{\lambda} & \text{ if $i$ is even and } b \ge2;\\
q^{r-1} +\displaystyle\prod_{j=1}^{u} \left(\frac{q}{p_j}\right)^{\max\{0,\nu_{p_j}(r)-\nu_{p_j}(i)\}}q^{\frac{(r+2^bv_0-2)}{2}}\varepsilon'_{\lambda} & \text{ if $i$ is odd and } b \ge2.\\
\end{cases}\]
%\begin{enumerate}[(i)]
%\item \(S_{\lambda} =q^{r -1} - \displaystyle \prod_{j=1}^{u} \left(\frac{q}{p_j}\right)^{\max\{0,\nu_{p_j}(\tilde{r})-\nu_{p_j}(i)\}}q^{\frac{r + 2 v_0 -2}{2}},\)
%  if $i$ is even and $b=1$
%\item 
%\(S_{\lambda} = q^{r-1} -(-1)^{(q^{\tilde{r}}-1)(2^b-v_1)/4}q^{\frac{(r-\tilde{r}v_1-2)}{2}}, \) if $i$ is even and $b\ge2$
%
%\item \(S_{\lambda}  = q^{r-1} -\prod_{j=1}^{u} \left(\frac{q}{p_j}\right)^{\max\{0,\nu_{p_j}(r)-\nu_{p_j}(i)\}}q^{\frac{(r+2^bv_0-2)}{2}}\varepsilon_{\lambda},\)  if $i$ is odd
%
%\end{enumerate}
where $v_0=\gcd(\tilde r,i)$ and $v_1 = \gcd(2^b,i)$.
\end{theorem}

\begin{proof} By Lemma \ref{dimensao} it follows that $v = \gcd(r,i)$. We split the proof in cases. 
\begin{enumerate}[(i)]
\item {\em Case $i$  even and $b=1$}. Using the  transitivity of the trace function,  we obtain
\begin{align} \label{ss1} \lambda = \Tr(\alpha^{q^i+1}-\alpha^2) = \Tr_{\F_{q^{2}}/\F_q}(\Tr_{\F_{q^{r}}/\F_{q^2}}(\alpha^{q^i+1}-\alpha^2)),\end{align}
where $\alpha $ is such that $\Phi_i(\alpha)=\lambda.$ Let $\mu = \Tr_{\F_{q^{r}}/\F_{q^2}}(\alpha^{q^i+1}-\alpha^2)\in \F_{q^2}$. Then  Equation \eqref{ss1} is equivalent to  the system
\[\left\{
\begin{array}{ll}
\Tr_{\F_{q^{2}}/\F_q}(\mu) = \lambda;\\
\Tr_{\F_{q^{r}}/\F_{q^{2}}}(\alpha^{q^i+1}-\alpha^2) = \mu.
\end{array}\right. \]
Let us define  $Q= q^{2}$, then $Q^{\tilde r} = q^r$. 
Since $b=1$, by Lemma \ref{rimpar} it follows that the number of solutions of $\Tr_{\F_{q^{r}}/\F_{q^2}}(x^{q^i+1}-x^2)) = \mu$ 
is $$Q^{\tilde{r} -1} - \displaystyle \prod_{j=1}^{u} \left(\frac{q}{p_j}\right)^{\max\{0,\nu_{p_j}(\tilde{r})-\nu_{p_j}(i)\}}Q^{\frac{\tilde r + v_0 -2}{2}}.$$
The dimension of the radical of   $\Tr_{\F_{q^2}/\F_{q}}(x^{q^i+1}-x^2))$ is $v_0 = \gcd(\tilde r,i)$. The number of solutions of  $\Tr_{\F_{q^2}/\F_q}(\mu) = \lambda$ is $q$  and therefore %for $\lambda \in \F_{q}^*$
\begin{align*}S_{\lambda} &= q \left(Q^{\tilde{r} -1} - \displaystyle \prod_{j=1}^{u} \left(\frac{q}{p_j}\right)^{\max\{0,\nu_{p_j}(\tilde{r})-\nu_{p_j}(i)\}}Q^{\frac{\tilde r + v_0 -2}{2}}\right)\\
& = q \left(q^{2\tilde{r} -2} - \displaystyle \prod_{j=1}^{u} \left(\frac{q}{p_j}\right)^{\max\{0,\nu_{p_j}(\tilde{r})-\nu_{p_j}(i)\}}q^{\frac{2\tilde r + 2 v_0 -4}{2}}\right)\\
& = q^{r -1} - \displaystyle \prod_{j=1}^{u} \left(\frac{q}{p_j}\right)^{\max\{0,\nu_{p_j}(\tilde{r})-\nu_{p_j}(i)\}}q^{\frac{r + 2 v_0 -2}{2}}.\\
\end{align*}

\item {\em Case $i$  even and $b\ge2$}.  As above, we have
\begin{align}  \lambda = \Tr(\alpha^{q^i+1}-\alpha^2) = \Tr_{\F_{q^{\tilde{r}}}/\F_q}(\Tr_{\F_{q^{r}}/\F_{q^{\tilde{r}}}}(\alpha^{q^i+1}-\alpha^2)),\label{s1}\end{align}
where $\alpha $ is such that $\Phi_i(\alpha)=\lambda.$
%Let $\mu \in \F_{q^{\tilde{r}}}^{*}$ such that  $\mu = \Tr_{\F_{q^r}/F_{q^{\tilde{r}}}}(x^{q^i+1}-x^2)$.
 Finding solutions of the Equation \eqref{s1} is equivalent to  finding  solutions  of the system
\[\left\{
\begin{array}{ll}
\Tr_{\F_{q^{\tilde{r}}}/\F_q}(\mu) = \lambda;\\
\Tr_{\F_{q^{r}}/\F_{q^{\tilde{r}}}}(\alpha^{q^i+1}-\alpha^2) = \mu.
\end{array}\right. \]
Putting $Q= q^{\tilde{r}}$ we have that $Q^{2^b} = q^r$.  Since $b\ge 2$, it follows from Lemma \ref{char2}(ii) that the number of solutions of  $\Tr_{\F_{q^{r}}/\F_{q^{\tilde{r}}}}(x^{q^i+1}-x^2)) = \mu$ is 
\[S_{\mu} = Q^{2^b-1} -(-1)^{(Q-1)\frac{(2^b-v_1)}{4}}Q^{\frac{(2^b-v_1-2)}{2}},\]
where $v_1 = \gcd(2^b,i)$ is the dimension of the radical of  $\Tr_{\F_{q^{r}}/\F_{q^{\tilde{r}}}}(x^{q^i+1}-x^2) $. Besides that, the number of solutions of  $\Tr_{\F_{q^{\tilde{r}}}/\F_q}(\mu) = \lambda$ is $q^{\tilde{r}-1}$ and  the number of solutions of  $\Phi_i(x) = \lambda$ is
\begin{align*}q^{\tilde{r}-1} (Q^{2^b-1} -(-1)^{(Q-1)\frac{(2^b-v_1)}{4}}Q^{\frac{(2^b-v_1-2)}{2}})&  =  q^{\tilde{r}-1} (q^{r-\tilde{r}} -(-1)^{q^{(\tilde{r}-1)\frac{(2^b-v_1)}{4}}}q^{\frac{(r-\tilde{r}v_1-2\tilde{r})}{2}})\\
& = q^{r-1} -(-1)^{q^{(\tilde{r}-1)\frac{(2^b-v_1)}{4}}}q^{\frac{(r-\tilde{r}v_1-2)}{2}}. 
\end{align*}
\item {\em Case  $i$  odd}. We define $Q= q^{2^b}$;  then $Q^{\tilde{r}} = q^r$ and 
%. Using the relation of transitivity of the trace function  we have
\begin{align} \label{s2}\lambda = \Tr(\alpha^{q^i+1}-\alpha^2) = \Tr_{\F_{Q}/\F_q}(\Tr_{\F_{q^{r}}/\F_{Q}}(\alpha^{q^i+1}-\alpha^2)),\end{align}
where $\alpha $ is such that $\Phi_i(\alpha)=\lambda.$ 
%Let  $\mu \in \F_{Q}^{*}$ such that  $\mu = \Tr_{F_{Q}/\F_q}(x^{q^i+1}-x^2)$.
It follows that the number of solutions of \eqref{s2} is equal to the number of solutions of the system 
\[\left\{
\begin{array}{ll}
\Tr_{\F_{Q}/\F_q}(\mu) = \lambda;\\
\Tr_{\F_{q^{r}}/\F_{Q}}(\alpha^{q^i+1}-\alpha^2)) = \mu.
\end{array}\right. \]
 By Corollary \ref{rimpar} we have that the number of solutions of  $\Tr_{\F_{q^{r}}/\F_{Q}}(x^{q^i+1}-x^2)) = \mu$ with $\mu\ne 0$ is  
\[Q^{\tilde{r}-1} - \displaystyle \prod_{j=1}^{u} \left(\frac{q}{p_j}\right)^{\max\{0,\nu_{p_j}(r)-\nu_{p_j}(i)\}}Q^{\frac{(\tilde{r}+v_0-2)}{2}},\]
where $v_0 = \gcd(\tilde{r},i)$ is the dimension of the radical of $\Tr_{\F_{q^{r}}/\F_{Q}}(x^{q^i+1}-x^2)$.
Since the number of solutions of  $\Tr_{\F_Q/\F_q}(\mu) = \lambda$ is $q^{2^b-1}$, we conclude that the number of solutions of $\Phi_i(x) = \lambda$ is 
\begin{align*}
S_{\lambda} & = q^{2^b-1}\left(Q^{\tilde{r}-1} - \displaystyle \prod_{j=1}^{u} \left(\frac{Q}{p_j}\right)^{\max\{0,\nu_{p_j}(r)-\nu_{p_j}(i)\}}Q^{\frac{(\tilde{r}+v_0-2)}{2}}\right) \\&= q^{2^b-1}\left(q^{r-2^b} - \displaystyle \prod_{j=1}^{u} \left(\frac{q}{p_j}\right)^{\max\{0,\nu_{p_j}(r)-\nu_{p_j}(i)\}}q^{\frac{(r+2^bv_0-2^{b+1})}{2}}\right)\\
& = q^{r-1} -\prod_{j=1}^{u} \left(\frac{q}{p_j}\right)^{\max\{0,\nu_{p_j}(r)-\nu_{p_j}(i)\}}q^{\frac{(r+2^bv_0-2)}{2}}.
\end{align*}
\end{enumerate}

The case $\lambda=0$ follows using the same ideas and   analogous  formulas for $S_0$, for extensions of $\F_q$ of degree  power of $2$ and odd.
\end{proof}

Using Lemma \ref{rimpar} and Theorem \ref{anyr} we can determine the number of affine rational points  of the curve $y^q-y = x^{q^i+1}-x^2-\lambda$, that we show as in the following  theorem.

\begin{theorem}\label{th2}
Let $i,r$ be integers such that $0< i < r$. Let $\tilde{r}$ be an integer such that  $r = 2^b\tilde{r}$, $\gcd(\tilde r,2p)=1$ and $ \tilde r=  p_1^{a_1}\cdots p_u^{a_u}$ the prime factorization. For  $F(x) = x^{q^i}-x$ and $\lambda \in \F_q$ we have  

%
%Let $i,r$ be integers such that $0< i < r$, $r = 2^b\tilde{r}$, $ \tilde r=  p_1^{a_1}\cdots p_u^{a_u}$ is the prime factorization \vml{and $ \tilde r=  p_1^{a_1}\cdots p_u^{a_u}$ the prime factorization} of $\tilde r$ satisfying  \vml{such that} $\gcd(\tilde r,2p)=1$. Let $\lambda \in \F_q$ and  $F(x) = x^{q^i}-x$. Then   
 
% ({\color{red} Na verdade, acho que pode comentar, porque esse caso é igual ao segundo item})
%\[(1-\chi(-\lambda))q^{\tilde{r}+1},\]
%if $b=1$ and $i$ is even. Otherwise, for $\lambda \in \F_q^*$ the number of affine rational points  is
$$N_r(\mathcal C_{F,\lambda}) = q^r+ Dq^{(r+L)/2}$$
%\varepsilon_{\lambda}$$  
where  
\begin{enumerate}[(i)]
\item D=  \((1-\chi(-\lambda))q^{\tilde{r}+1}\varepsilon_{\lambda}, $ $  L=-\tilde{r}\gcd(2^b,i)\)
 if $b=1$ and $i$ is even; 
\item \(D = (-1)^{(q%^{t_1^{a_1}\cdots t_u^{a_u}}
-1)(2^b-\gcd(2^b,i))/4} \varepsilon_{\lambda}, $ $ L=-\tilde{r}\gcd(2^b,i)\)  if $b\ge2$ and $i$ is even;
\item \(D =  \prod_{j=1}^{u} \left(\frac{q}{p_j}\right)^{\max\{0,\nu_{p_j}(r)-\nu_{p_j}(i)\}}\varepsilon'_{\lambda}, $ $  L=2^b\gcd(\tilde{r},i)\)  if  $b=0$ or $b\ge 1$  and $i$ is odd. 
\end{enumerate}

\end{theorem}

\section{Some open problems}
We finish this paper enumerating some open problems.
We note that in Theorem \ref{thF(x)} we determined the number of  affine rational points   in  $\F_{q^r}^2$ of  $\mathcal C:y^q-y = x  F(x) - \lambda$ when $F(x)$ is a $\F_q$ linearized  and such that $g(x) = \gcd(f(x),x^r-1)$ is a self-reciprocal, where $f(x)$ is the associated to $F(x)$. We then have two problems:

\begin{problem}\label{p1}
Determine $N_r(\mathcal C_{F,\lambda})$ when $F(x)$ is not a $\F_q$-linearized.
\end{problem}

\begin{problem}\label{p2}
Determine $N_r(\mathcal C_{F,\lambda})$ when $g(x)$ is not a self-reciprocal.
\end{problem}

In Section $4$, we only considered extensions of degree $r$ such that $\gcd(r,p) =1$. This is the content of the next problem.
\begin{problem}\label{p3}
Determine explicitly $N_r(\mathcal C_{x  F(x), \lambda})$, when $\lambda \in \F_q$ and $F(x)$ is a $\F_q$ linearized  and such that $\gcd(r,p) = p$.
\end{problem}
In \cite{matilde}, the authors show that the number of monic irreducible polynomials in $\F_q[x]$ of degree $r$ and with the first and third coefficients prescribed is related to the curve 
\[y^q-y = x^{2q+1}-x^{q+2}.\]
Then, we have the following problem.
\begin{problem}\label{p4}
Determine $N_r(\mathcal C_{FG})$ when $F(x),G(x) \in \F_q[x]$ are $\F_{q}$-linearized.
\end{problem}

%
%
%\begin{problem}\label{p5}
%Determine $N_r(\mathcal C_{FG})$ when $F(x),G(x) \in \F_q[x]$ are $\F_{q}$-linearized.
%\end{problem}

\end{document}